\documentclass[11pt,reqno]{amsart}
\usepackage{amsmath}
\usepackage{amsthm}
\usepackage{graphicx}
\usepackage{amsfonts}
\usepackage{bm}
\usepackage{bbm}
\usepackage{textcomp}
\usepackage{hyperref}
\usepackage{amssymb}
\usepackage{mathrsfs}
\usepackage{enumerate}
\usepackage[margin=1.1in]{geometry}
\usepackage{xcolor}

\numberwithin{equation}{section}

\newtheorem{theorem}{Theorem}[section]
\newtheorem{lemma}[theorem]{Lemma}
\newtheorem{proposition}[theorem]{Proposition}

\theoremstyle{definition}

\newtheorem{definition}[theorem]{Definition}
\newtheorem{remark}[theorem]{Remark}

\newtheorem*{standingassumption}{Standing assumptions}

\def\d{{\mathrm d}}
\def\E{{\mathbb E}}

\def\R{{\mathbb R}}

\def\N{{\mathbb N}}
\def\PP{{\mathbb P}}
\def\FF{{\mathbb F}}
\def\P{{\mathcal P}}

\def\GG{{\mathcal G}}

\def\T{{\mathbb T}}
\def\TT{{\mathcal T}_m}

\def\Z{{\mathbb Z}}

\def\F{{\mathcal F}}

\def\Var{{\mathrm{Var}}}
\def\Cov{{\mathrm{Cov}}}
\def\newW{K}
\def\K{\newW}

\title[Stationary local equations on regular trees]{Stationary solutions and local equations for interacting diffusions on regular trees}
\author{Daniel Lacker and Jiacheng Zhang}
\address{Department of Industrial Engineering \& Operations Research, Columbia University}
\email{daniel.lacker@columbia.edu}
\address{Department of Industrial Engineering \& Operations Research, University of California, Berkeley}
\email{jiachengz@berkeley.edu}
\thanks{D.L. is partially supported by the AFOSR Grant FA9550-19-1-0291 and the NSF CAREER award DMS-2045328.}

\begin{document}

\begin{abstract}
We study the invariant measures of infinite systems of stochastic differential equations (SDEs) indexed by the vertices of a regular tree. These invariant measures correspond to Gibbs measures associated with certain continuous specifications, and we focus specifically on those measures which are homogeneous Markov random fields.
We characterize the joint law at any two adjacent vertices in terms of a new two-dimensional SDE system, called the ``local equation," which exhibits an unusual dependence on a conditional law.
Exploiting an alternative characterization in terms of an eigenfunction-type fixed point problem, we derive existence and uniqueness results for invariant measures of the local equation and infinite SDE system.
This machinery is put to use in two examples. First, we give a detailed analysis of the surprisingly subtle case of linear coefficients, which yields a new way to derive the famous Kesten-McKay law for the spectral measure of the regular tree. Second, we construct solutions of tree-indexed SDE systems with nearest-neighbor repulsion effects, similar to Dyson's Brownian motion.
\end{abstract}

\maketitle

\section{Introduction}

In this paper we propose new characterizations for stationary solutions of infinite systems of stochastic differential equation indexed by the $m$-regular (Cayley) tree $\T_m$. For an integer $m \ge 2$, $\T_m$ is the infinite tree in which each vertex has exactly $m$ neighbors. 
For a vertex $v \in \T_m$, let $N(v)$ denote the set of adjacent vertices (neighbors).
For functions $U,\K : \R \to \R$, to be specified later, we consider $\bm X:=(X^v)_{v \in \T_m}$ satisfying the stochastic differential equation (SDE) 
\begin{equation}
\d X^v_t = -\bigg(U'(X^v_t) + \sum_{u \in N(v)}\K'(X^v_t-X^u_t)\bigg) \,\d t +\sqrt{2}\,\d W^v_t, \qquad v \in \T_m, \label{intro:SDE}
\end{equation}
where $\bm W:=(W^v)_{v \in \T_m}$ are independent Brownian motions.

The same SDE \eqref{intro:SDE} can be defined on any finite (or countable, locally finite) graph $G$ in place of $\T_m$, and we refer naturally to \emph{the SDE system set on $G$}.
Our study of this SDE systems builds on a recent line of literature on scaling limits of SDEs with nearest-neighbor interactions on large, finite graphs.
For dense Erd\H{o}s-R\'enyi graphs, the large-$n$ behavior is nearly identical to the mean field (complete graph) case \cite{bhamidi2019weakly,delattre2016note,
coppini2020law,oliveira2019interacting}. More generally, SDE systems set on dense graphs which converge in the graphon sense still exhibit certain mean field behaviors such as asymptotic independence  \cite{medvedev2014nonlinear,bet2020weakly,bayraktar2020graphon}. On the other hand, the case of \emph{sparse} graphs is not as well understood.
The recent results of \cite{oliveira2020interacting,lacker2019large}
show that the SDE system set on $\T_m$ arises as a suitable \emph{local} limit of finite systems: If the graph $\T_m$ is replaced by a sequence of finite graphs converging in the local weak (a.k.a.\ Benjamini-Schramm) sense to $\T_m$, such as the uniformly random $m$-regular graph on $n$ vertices as $n\to\infty$, then the corresponding SDE solutions converge in a certain local sense to \eqref{intro:SDE}, under suitable assumptions on the coefficients and initial conditions.
The paper \cite{lacker2020marginal} goes a step further by characterizing the marginal dynamics of any neighborhood in $\T_m$, in terms of what they call the \emph{local equation}, described in detail in Section \ref{se:localeqs} below.

The results of \cite{oliveira2020interacting,lacker2019large}, as well as most of the others discussed in the previous paragraph, pertain to \emph{dynamics} (out of equilibrium), and the goal of the present paper is to carry out an analysis of \emph{stationary solutions} (invariant measures). Our main results will characterize the marginal law of adjacent vertex pairs in stationary solutions of \eqref{intro:SDE} in terms of a \emph{local equation}, which is analogous to the dynamic local equation of \cite{lacker2020marginal}  but much more tractable, as we will see.
Understanding invariant measures is a natural problem in its own right, and it is also requisite first step toward an analysis of long-time behavior and metastability phenomena, which we do not carry out in this work.
Let us mention also that when $m=2$ or when the graph $\T_m$ is replaced by an integer lattice $\Z^d$, SDEs of the form \eqref{intro:SDE} arise in the study of the \emph{Ginzburg-Landau $\nabla \phi$ interface models}; see \cite{funaki-interface} for an overview.

If the SDE system \eqref{intro:SDE} is set on a general \emph{finite} graph $G=(V,E)$, then, under modest assumptions on $(U,K)$, the unique invariant measure (defined on $\R^V$) is  given by
\begin{align}
\mu_G(\d \bm x)=\frac{1}{Z}\exp\bigg( -\sum_{v \in V} U(x_v) - \sum_{\{u,v\} \in E} \K(x_v-x_u)\bigg)\,\d \bm x, \qquad Z > 0, \label{intro:Gibbs}
\end{align}
the summation being over undirected edges.
On an \emph{infinite} graph, there is no explicit density like \eqref{intro:Gibbs} to work with, and analogous measures may be studied using  the formalism of Gibbs measures  \cite{georgii2011gibbs,rassoul2015course}. Section \ref{se:gibbsconnection} will explain how the Gibbs measure formalism relates to the invariant measures of \eqref{intro:SDE}. 

There are two useful immediate properties of $\mu_G$, for a finite graph $G$. First, it is a Markov random field, which follows quickly from the factorized form of the density: If $\bm X_0:=(X^v_0)_{v \in V} \sim \mu_G$, and if $A,B,S \subset V$ are disjoint sets with the property that every path from $A$ to $B$ has at least one vertex in $S$, then $\bm X^A_0:=(X^v_0)_{v \in A}$ and $\bm X^B_0$ are conditionally independent given $\bm X^S_0$.
The second useful property is \emph{homogeneity} or \emph{automorphism invariance}: For any automorphism $\varphi$ of $G$, we have $(X^{\varphi(v)}_0)_{v \in V} \stackrel{\d}{=} (X^v_0)_{v \in V}$.
Our main interest in this paper is in the invariant measures of \eqref{intro:SDE} possessing a form of these Markov and homogeneity properties.

\subsection{Main results}

 In the following, we write $\bm x=(x_v)_{v\in\T_m}$ for a generic element of $\R^{\T_m}$, and $\bm X=(X^v)_{v\in\T_m}$ for the stochastic process of interest. We write also $\bm x_S=(x_v)_{v\in S}$ and $\bm X^S=(X^v)_{v\in S}$ for the restrictions to coordinates in a set $S \subset \T_m$. 
We make the following assumptions throughout the paper:

\begin{standingassumption}
{We are given an integer $m \ge 2$ and a.e.\ finite functions $U,\K : \R \to \R \cup \{\infty\}$ such that $K$ is even and $e^{-U},e^{-\K} : \R \to [0,\infty)$ are locally absolutely continuous. In particular, $U$ and $\K$ are a.e.\ differentiable.}
\end{standingassumption}

\begin{definition}\label{def:infiniteSDE}
A \emph{stationary solution of the infinite SDE} is a weak solution $\bm X:=(X^v)_{v \in \T_m}$  of \eqref{intro:SDE} which satisfies the time-stationarity property $\bm X_t \stackrel{\d}{=} \bm X_0$ for all $t \ge 0$, as well as
\begin{align}
\E\big[ |U'(X^v_0)| + |\K'(X^v_0-X^u_0)| \big] < \infty, \label{asmp:infSDE-integ}
\end{align}
for every edge $(v,u)$.
A \emph{stationary homogeneous Markov (SHM) solution of the infinite SDE} is a stationary solution which satisfies two additional properties:
\begin{itemize}
\item Homogeneity: $(X^{\varphi(v)}_0)_{v \in V} \stackrel{\d}{=} (X^v_0)_{v \in V}$ for each automorphism $\varphi$ of $\T_m$.
\item Markov property: Let $A,B,S \subset \T_m$ be finite sets such that every path from $A$ to $B$ passes through $S$. Then $\bm X^A_0$ and $\bm X^B_0$ are conditionally independent given $\bm X^S_0$.
\end{itemize}
\end{definition}

Note that the condition \eqref{asmp:infSDE-integ} ensures that the SDE \eqref{intro:SDE} makes sense.
We will characterize SHM solutions of the infinite SDE in terms of a new two-dimensional SDE system which describes the behavior of a single edge. We call it the \emph{local equation}, in analogy with the recent work \cite{lacker2020marginal} in the non-stationary regime, discussed further in Section \ref{se:relatedwork}.

\begin{definition}\label{def:localequation}
A \emph{stationary symmetric solution of the local equation} is a weak solution $(X,Y)$ of the SDE system
\begin{align}
\begin{split}
\d X_t &= -\big(U'(X_t) + \K'(X_t-Y_t) + (m-1)\E[\K'(X_t-Y_t)\,|\,X_t]\big)\d t + \sqrt{2} \, \d W_t, \\
\d Y_t &= -\big(U'(Y_t) + \K'(Y_t-X_t) + (m-1)\E[\K'(Y_t-X_t)\,|\,Y_t]\big)\d t + \sqrt{2} \,\d B_t,
\end{split} \label{intro:localeq}
\end{align}
where $W$ and $B$ are independent Brownian motions, which is stationary in the sense that $(X_t,Y_t) \stackrel{\d}{=} (X_0,Y_0)$ for all $t \ge 0$, and symmetric in the sense that $(X_0,Y_0) \stackrel{\d}{=} (Y_0,X_0)$. We implicitly require in this definition that
\begin{align}
\E\big[ |U'(X_0)| + |\K'(X_0-Y_0)| \big] < \infty. \label{asmp:localeq-integ}
\end{align}
\end{definition}

The SDE \eqref{intro:localeq} is nonlinear in the sense of McKean (cf.\ \cite{sznitman1991topics}), as the drift depends on the joint law of the solution $(X_t,Y_t)$. The particular dependence on the \emph{conditional} law, which is not weakly continuous as a functional of the joint law, places \eqref{intro:localeq} outside the scope of standard theory. Section \ref{se:SDEscondexp} discusses  some SDEs with similar conditional expectation terms that arise in different contexts.

The main message of the paper is that there is essentially a one-to-one correspondence between SHM solutions of the infinite SDE and stationary symmetric solutions of the local equation. Before we state precise results, we introduce a third formulation, which is again essentially equivalent to the above two. It will be useful in the proofs as well as some of the examples to follow.

\begin{definition} \label{def:fixedpoint}
A \emph{solution of the fixed point problem} is a measurable function $F : \R \to \R$ such that there exists $C \in \R$ satisfying 
\begin{align}
&F(x) = C -\log\int_\R e^{-U(y)-\K(x-y)-(m-1)F(y)}\d y, \ \ a.e.\ x \in \R, \text{ and}  \label{fixedpoint} \\
&\int_\R e^{-U(x)-mF(x)}\,\d x < \infty. \label{fixedpoint-integrability}
\end{align}
Note that the integral in \eqref{fixedpoint} always exists in $(0,\infty]$, and it must in fact be finite in order to have a solution of the fixed point problem.
\end{definition}

The fixed point problem can be interpreted as determining the invariant measures for a $\T_m$-indexed Markov chain. In the $m=2$ case, letting $\varphi=e^{-F}$, we see that $F$ solves the fixed point problem if and only if $\varphi$ is a (positive) eigenfunction of the integral operator $\varphi(\cdot) \mapsto \int_\R \varphi(y)\exp(-U(y)-\K(\cdot\, -y))\,\d y$.

Our three main theorems relate the above notions. First, moving from the infinite SDE to the local equation requires no additional assumptions:

\begin{theorem} \label{th:infSDE->local}
Let $(X^v)_{v \in \T_m}$ be a SHM solution of the infinite SDE. Then there exists a stationary symmetric solution $(X,Y)$ of the local equation such that $(X_0,Y_0) \stackrel{d}{=} (X^v_0,X^u_0)$ for each edge $(v,u)$ in $\T_m$.
\end{theorem}

The converse, constructing an SHM solution from a solution of the local equation, is more involved and requires some additional assumptions of a technical nature. Our proof uses the fixed point problem along the way, so we split this direction into two theorems.

\begin{theorem} \label{th:local->fixedpoint}
Let $(X,Y)$ be a stationary symmetric solution of the local equation. Define
\begin{align*}
f(x) := \E[\K'(X_0-Y_0)\,|\,X_0=x].
\end{align*}
Suppose $U'$, $\K'$, and $f$ belong to $L^q_{\mathrm{loc}}(\R)$ for some $q > 2$. Then $F(x) := \int_0^x f(y)\,\d y$ is a solution of the fixed point problem. Moreover, the law of $(X_0,Y_0)$ admits the density 
\begin{equation}\label{jointdensity-thm}
\rho(x,y) = Z^{-1}\exp\big(-U(x)-U(y)-K(x-y)-(m-1)F(x)-(m-1)F(y)\big),
\end{equation}
where $0 < Z < \infty$ is defined by
\begin{align}
Z := \int_{\R^2}\exp\Big( -U(x)-U(y)-\K(x-y) - (m-1)F(x)-(m-1)F(y)\Big)\,\d x \d y. \label{jointdensity-normalization}
\end{align}
\end{theorem}

\begin{theorem} \label{th:fixedpoint->infSDE}
Let $F$ be a solution of the fixed point problem satisfying
\begin{align}
\int_{\R^2} (|U'(x)|^p + |\K'(x-y)|^p)e^{-U(x)-U(y)-K(x-y)-(m-1)F(x)-(m-1)F(y)}\,\d x \d y < \infty, \label{asmp:tech1'}
\end{align}
for some $p > 1$. Then $Z$ defined by \eqref{jointdensity-normalization} is finite, and there exists a SHM solution $(X^v)_{v \in \T_m}$ of the infinite SDE such that, for each edge $(u,v)$, the law of $(X^v_0,X^u_0)$ is given by the density function \eqref{jointdensity-thm}.
If it holds also that
\begin{align}
\int_A \bigg(\int_{\R} |\K'(x-y)|^m e^{-U(y)-m \K(x-y)} \d y \bigg)^{1/m} \d x &< \infty \label{asmp:tech2} 
\end{align}
for every compact set $A \subset \R$, then $F$ is locally absolutely continuous with $F'(x)=\E[\K'(X^v_0-X^u_0)\,|\,X^u_0=x]$ for a.e.\ $x \in \R$, for any edge $(v,u)$.
\end{theorem}

The proofs are given in Section \ref{se:mainthmproofs}.
Theorem \ref{th:infSDE->local} comes from applying the so-called \emph{mimicking theorem} \cite{gyongy1986mimicking,brunick2013mimicking}, recalled in Theorem \ref{th:mimicking} below, and using the Markov property and homogeneity to simplify. Theorem \ref{th:local->fixedpoint} is a consequence of direct calculations involving the explicit form of the density. The local integrability assumption therein ensures that the probability density \eqref{jointdensity-thm}
is the \emph{unique} solution of an associated stationary Fokker-Planck equation, using results from \cite{bogachev2015fokker}. It seems plausible that the local integrability assumption could be worked around on a case-by-case basis.

Theorem \ref{th:fixedpoint->infSDE} is the most difficult.
{It is straightforward to argue that the law of a homogeneous Markov random field on $\T_m$ can always be reconstructed from the joint law at two adjacent vertices; in the $m=2$ case, this is just the construction of the law of a  two-sided reversible Markov chain $(Z_i)_{i \in \Z}$ from the law of $(Z_0,Z_1)$. To prove Theorem \ref{th:fixedpoint->infSDE}, we apply this construction starting from the joint density \eqref{jointdensity-thm}, and the main difficulty is to show that the resulting random field on $\T_m$ agrees in law with a SHM solution.} This exploits a related but distinct \emph{boundary law} representation for Markov chains on trees \cite[Chapter 12]{georgii2011gibbs}, discussed  in Section \ref{se:boundarylaws}.

\begin{remark} \label{re:tech1}
Suppose that for some $p > 1$ it holds that
\begin{align}
\int_{\R^2} (|U'(x)|^{mp}+|\K'(x-y)|^{mp}) e^{-U(x)-U(y)-m \K(x-y)} \,\d x \, \d y &< \infty. \label{asmp:tech1}
\end{align}
Then \eqref{asmp:tech1'} is automatically satisfied by any solution of the fixed point problem. Indeed, by H\"older's inequality, the left-hand side of \eqref{asmp:tech1'} is bounded by a constant times
\begin{align*}
&\bigg(\int_{\R^2} (|U'(x)|^{mp} +|\K'(x-y)|^{mp})e^{-U(x)-U(y)-mK(x-y)}\,\d x\d y\bigg)^{\frac1 m}\\
	&\quad \cdot \bigg(\int_{\R^2} e^{-U(x)-U(y)-mF(x)-mF(y)}\d x \d y \bigg)^{\frac{m-1}{m}}.
\end{align*}
The first term is finite because of \eqref{asmp:tech1}, and the second because of \eqref{fixedpoint-integrability}.
\end{remark}

\subsection{Well-posedness}

The above theorems allow us to transfer existence and uniqueness results between the infinite SDE, local equation, and fixed point problem.
Our next results, proven in Section \ref{se:wellposedness}, provide some  existence and uniqueness theorems by focusing on the fixed point problem, which is the most amenable to well-posedness analysis.. For $m=2$ (i.e., $\T_2=\Z$), existence and uniqueness hold under no additional assumptions, which is not surprising given that Gibbs measures on $\Z$ are typically unique \cite[Chapter 10]{georgii2011gibbs}.

\begin{theorem} \label{th:intro:wellposed-m=2}
Suppose $m=2$. 
\begin{itemize}
\item If \eqref{asmp:tech2} holds (with $m=2$), then there is at most one solution of the fixed point problem up to additive shifts. If also $U',K'\in L^q_{\text{loc}}(\R)$ for some $q>2$, then there is at most one stationary symmetric solution or the local equation (resp.\ SHM solution of the infinite SDE) among those  satisfying $f\in L^q_{\text{loc}}(\R)$, where
\begin{align} 
f(x) &:=\E[K'(X_0-Y_0)|X_0=x],  \label{f-def-local} \\
(\text{resp. } f(x) &:=\E[K'(X^v_0-X^u_0)|X_0=x] \text{ for some edge } (v,u)  ).  \label{f-def-SHM}
\end{align}
\item If 
\begin{align}
\int_\R \int_\R e^{-U(x)-U(y)-2\K(x-y)}\, \d x \d y < \infty, \label{asmp:m=2}
\end{align}
then there exists a solution of the fixed point problem. If also \eqref{asmp:tech1} holds with $m=2$ for some $p > 1$, then there exist a stationary symmetric solution of the local equation and an SHM solution of the infinite SDE.
\end{itemize}
\end{theorem}

For $m > 2$, well-posedness is more delicate and cannot be expected to hold in general, as phase transitions are more prevalent for Gibbs measures on $\T_m$. 

\begin{theorem} \label{th:intro:wellposed-m>2}
Suppose $m > 2$. Assume $U$ and $\K$ are even functions and that $U'$ and $\K'$ are absolutely continuous, with $\K'' \in L^\infty(\R)$ and
\begin{align}
\mathrm{ess}\inf_x U''(x) > m\big(\|\K''\|_\infty-\mathrm{ess}\inf_x\K''(x)\big). \label{asmp:m>2-thm}
\end{align}
Then there exists a unique solution of the fixed point problem up to additive shifts. Uniqueness holds for the stationary symmetric solution of the local equation (resp.\ SHM solution of the infinite SDE) among the class of solutions satisfying $f \in L^q_{\mathrm{loc}}(\R)$ for some $q>2$, where $f$ is given as in \eqref{f-def-local} (resp.\ \eqref{f-def-SHM}). {If also there exist $0 \le p < 2$  and $c_1,c_2 > 0$ such that $|U'(x)| \le c_1e^{c_2|x|^p}$, then there exist a stationary symmetric solution of the local equation and a SHM solution of the infinite SDE}.
\end{theorem}

The assumption \eqref{asmp:m>2-thm} in Theorem \ref{th:intro:wellposed-m>2} requires the convexity of the confinement potential to be strong relative to the interaction potential. Note that the right-hand side of \eqref{asmp:m>2-thm} is nonnegative, so $U$ is necessarily convex, but $\K$ need not be convex.
When $\K$ is quadratic, the assumption \eqref{asmp:m>2-thm} is simply the uniform convexity of $U$.
A sufficient condition for \eqref{asmp:m>2-thm} is that $U'' \ge 2m\|\K''\|_\infty$ pointwise, where we note that $\|\K''\|_\infty$ is the Lipschitz constant of $\K'$. {Convexity is well known to play an important role in the uniqueness on invariant measures (ground states) for mean field models \cite{mccann1997convexity,carrillo2003kinetic}, and a natural open question is if mere uniform convexity of $U$ and $\K$ would suffice for uniqueness in the setting of Theorem \ref{th:intro:wellposed-m>2}. Smallness conditions on the interaction potential, similar to \eqref{asmp:m>2-thm}, are also known ensure uniqueness in mean field models when the interaction is not required to be convex  \cite{guillin2019uniform,lacker2021quantitative}.}

\subsection{Examples}
After proving the above general theorems, the rest of the paper is devoted to two noteworthy examples, which we summarize here.

\subsubsection{Linear coefficients and the resolvent of $\T_m$} \label{se:intro:linear}

In Section \ref{se:linear-coeff}, we study the case of linear coefficients. Specifically, consider $U(x)=(z-m)x^2/2$ and $\K(x)=x^2/2$, for some fixed $z \in \R$. Then the SDE system \eqref{intro:SDE} can be written in terms of $\bm{X}_t = (X^v_t)_{v \in \T_m}$ as
\begin{align*}
\d \bm{X}_t &= -(z\bm{I} - \bm{A})\bm{X}_t \d t + \sqrt{2}\, \d \bm{W}_t,
\end{align*}
where $\bm{A}$ is the \emph{adjacency operator} of the tree, and $\bm{I}$ is the identity operator; see Section \ref{se:linear-coeff} for precise definitions.
Formally, the invariant measure should be an infinite-dimensional centered Gaussian with covariance operator $(z\bm{I}-\bm{A})^{-1}$, which is the resolvent of $\bm{A}$, assuming the inverse exists, say in the Hilbert space $\ell_2(\T_m)$.
Letting $(\bm e_v)_{v \in \T_m}$ denote the standard basis of $\ell_2(\T_m)$, a formal computation shows that $\langle \bm e_v,(z\bm{I}-\bm{A})^{-1}\bm e_v\rangle = \Var(X^v_t)$.
For $z > m$, we then show how to explicitly compute this variance by solving the local equation, which gives an expression for $\Var(X^v_t)$ in terms of $z$. This yields a formula for the resolvent $z \mapsto \langle \bm e_v,(z\bm{I}-\bm{A})^{-1}\bm e_v\rangle$, which is exactly the Cauchy-Stieltjes transform of the spectral measure of $\bm{A}$.
This recovers, by a new method, the famous Kesten-McKay law \cite{kesten1959symmetric,mckay1981expected} for the $m$-regular tree $\T_m$.

There is an interesting additional regime where $2\sqrt{m-1}<z<m$. In this case, there are in fact two Gaussian SHM solutions. One is distinguished by its lower edge-wise correlations and good summability  properties, which allows it to again be connected to the resolvent as above; essentially, it can be viewed as a Gaussian process on $\ell_2(\T_m)$. The second solution has higher correlations, and it cannot be viewed as a process on $\ell_2(\T_m)$.

\subsubsection{Particle systems with repulsion} \label{se:intro:repulsion}
Lastly, in Section \ref{se:repulsion}, we study a model inspired by the $\beta$-ensembles of random matrix theory. Specifically, consider $\K(x)=-\beta\log|x|$, for $\beta > 0$.
The resulting SDE system has a repulsive drift as in the SDE for Dyson's Brownian motion:
\begin{align}
\d X^v_t &= \left(- U'(X^v_t) + \sum_{u \in N(v)}\frac{\beta}{X^v_t-X^u_t}  \right)\d t + \sqrt{2} \, \d B^v_t, \quad v \in \T_m. \label{DysonSDE}
\end{align}
In the case where $\beta=2$ and $U$ is even, we use the fixed point problem to construct an SHM solution of \eqref{DysonSDE} with an explicit joint density for any pair of adjacent particles. The solution of the fixed point problem is unique, but we do not know if the SHM solution is, because the singularity of $\K'$ prevents an application Theorem \ref{th:local->fixedpoint}.
{We restrict to $\beta=2$ to enable exact calculations which are unavailable for general $\beta$, which is natural in analogy with the special tractability of the $\beta=2$ (GUE) ensemble in random matrix theory. A surprise in our setting is that explicit solutions are still available for general, non-quadratic potentials $U$.}

On a finite graph $G=(V,E)$ in place of $\T_m$, the invariant measure of \eqref{DysonSDE} corresponds to what one might call a \emph{$\beta$-ensemble on $G$}, given by the probability measure
\begin{align*}
\mu_G(d\bm{x}) &= \frac{1}{Z} \prod_{\{u,v\} \in E}|x_u-x_v|^\beta  \prod_{v \in V}e^{-U(x_v)}\d x_v.
\end{align*}
In two extreme cases, this corresponds to a random matrix model:
When $G$ is a complete graph, and when $U(x)=-\beta x^2/4$, this is precisely the well-known $\beta$-ensemble of random matrix theory. For instance, in the case $\beta=1$ (resp.\ $\beta=2$), the measure $\mu_G$ becomes the (symmetrized) joint law of the eigenvalues of a GOE (resp.\ GUE) random matrix \cite[Section 2.5]{anderson2010introduction}.  On the other extreme, if the graph is the trivial one (no edges), then the particles are i.i.d.\ and we can think of them as the eigenvalues of a random diagonal matrix with i.i.d.\ entries.
For other graphs $G$ between these extremes, it is not clear but natural to wonder if $\mu_G$ corresponds similarly to the eigenvalue distribution of any natural random matrix model.

\subsection{On the dynamic case}

It is worth explaining why we solely consider \emph{time-stationary} solutions of the SDEs \eqref{intro:SDE} and \eqref{intro:localeq} in this paper. The \emph{Markov property} is essential in deducing the local equation \eqref{intro:localeq} from the infinite SDE \eqref{intro:SDE}. For stationary solutions, it is natural to expect that the Markov property holds, as it does when the model is restricted to finite graphs as discussed around \eqref{intro:Gibbs}. In a non-stationary context, however, one cannot expect that the solution of \eqref{intro:SDE} to define a Markov random field at each time $t$, even on a finite graph. One reason is because of the phenomenon of \emph{Gibbs-non-Gibbs transitions} \cite{van2010gibbs,van2009gibbsianness}. More recently, it was shown in \cite{lacker2021locally} that the natural conditional independence structure in the non-stationary setting is that the \emph{trajectories} $((X^v_s)_{s \le t})_{v \in \T_m}$ form a \emph{second-order} Markov random field (assuming the same for the initial conditions).
In short, the non-stationary version of \eqref{intro:SDE} will rarely if ever have the property that $(X^v_t)_{v \in \T_m}$ is a Markov random field for each $t \ge 0$, even if it is initialized as such, and this breaks the connection with the local equation \eqref{intro:localeq}.

\subsection{Related literature} \label{se:relatedwork}

\subsubsection{Local equations} \label{se:localeqs}
The SDE system \eqref{intro:localeq} is the stationary analogue of the \emph{local equation} introduced recently in the dynamic setting in \cite{lacker2020marginal}, as part of a program to understand scaling limits of interacting diffusions on large (locally convergent) sparse graphs \cite{lacker2019large}. 
The main the result of \cite{lacker2020marginal} applies more generally to unimodular Galton-Watson trees, but in the case of a non-random tree it takes the following form. Suppose that an  i.i.d.\ initial distribution $(X^v_0)_{v \in \T_m}$ is prescribed, and that $U'$ and $\K'$ are globally Lipschitz, so that the SDE system is well-posed for all time. For a distinguished \emph{root vertex} denoted $0 \in \T_m$ and its neighborhood $\{1,\ldots,m\}$, it is shown in \cite[Section 3.2.1]{lacker2020marginal} that $(X^i_\cdot)_{i=0}^m \stackrel{\d}{=} (Y^i_\cdot)_{i=0}^m$, where $(Y^i_\cdot)_{i=0}^m$ is the unique in law solution of the \emph{local equation}
\begin{align*}
\d Y^{0}_t &= -\Big(U'(Y^{0}_t) + \sum_{v =1}^m\K'(Y^{0}_t-Y^v_t)\Big)\d t + \sqrt{2}\, \d W^{0}_t, \\
\d Y^v_t &= -\Big(U'(Y^v_t) + \K'(Y^v_t-Y^{0}_t) + (m-1)\gamma_t(Y^v,Y^{0}) \Big)\d t + \sqrt{2}\, \d W^v_t, \  v=1,\ldots,m \\
\gamma_t(y^{0},y^1) &= \E\big[\K'(Y^2_t - Y^{0}_t)\,\big|\,Y^{0}_s=y^{0}_s,\,Y^1_s=y^1_s, \, s \le t\big],
\end{align*}
where $(W^v)_{v =0}^m$ are independent Brownian motions.
The appearance of a full neighborhood, rather than a single edge as in our \eqref{intro:localeq}, is due to the fact that the trajectories of the dynamic particle system form a \emph{second-order} Markov random field, rather than \emph{first-order} as in our case. Similarly, the conditional expectation term $\gamma_t$ depends on the entire trajectories of $(y^0,y^1)$ rather than just the time-$t$ values, because it is the former and not the latter which form a second-order Markov random field.

The above \emph{local equation} from \cite{lacker2020marginal} describes the out-of-equilibrium dynamics of a single neighborhood of \eqref{intro:SDE}, whereas our new local equation \eqref{intro:localeq} describes the equilibrium for a single edge of \eqref{intro:SDE}. It would be natural to try to connect the two, for instance, by trying to analyze the $t \to\infty$ behavior of the system \eqref{intro:SDE} in terms of the local equations. We do not pursue this here.

\subsubsection{SDEs with conditional expectations} \label{se:SDEscondexp}

A peculiar feature of the local equation \eqref{intro:localeq} is that the coefficients depend on a conditional expectation. Related equations have appeared in different contexts, including simulation \cite{JLR}, Lagrangian stochastic models of  turbulent flows \cite{bossy2011conditional}, and local stochastic volatility models in mathematical finance \cite{lacker2020inverting,jourdain2020existence}.
Our paper shows yet another context in which SDEs with conditional expectations naturally arise, and our existence and uniqueness for stationary solutions of the local equation are not covered by these prior works.

\subsubsection{Gibbs measures on Cayley trees}

There is an extensive literature on Gibbs measures on the infinite (Cayley) trees $\T_m$, where many interesting phase transitions are remarkably tractable compared to more traditional lattice models. 
The classic book \cite{georgii2011gibbs} gives a thorough treatment of Gibbs measures in general, with Chapter 12 therein giving an overview of Markov random fields on trees, including results of the important early papers \cite{spitzer1975markov,zachary1983countable}. The property we call \emph{homogeneity} is also known as \emph{translation invariance} in various references, and Gibbs measures having the \emph{Markov property} are also known as \emph{splitting Gibbs measures}.
In contrast with the present paper, the vast majority of this literature deals with a discrete state (spin) space instead of $\R$. While we cannot do justice to this literature here, we refer to the monograph \cite{rozikov2013gibbs} and the recent survey \cite{rozikov2021gibbs} focused on Potts models for a thorough overview.
Our use of the fixed point problem (Definition \ref{def:fixedpoint}) is similar to the use of \emph{boundary laws} in the study of homogeneous Markov Gibbs measures, recalled in Section \ref{se:boundarylaws} below, but we are not aware of any prior results analogous to our connection between the infinite SDE and local equation.

\subsection{Organization of the paper}

In Section \ref{se:mainthmproofs} we prove the three main Theorems \ref{th:infSDE->local}, \ref{th:local->fixedpoint}, and \ref{th:fixedpoint->infSDE}. Section \ref{se:gibbsconnection} discusses the relationship between invariant measures of the infinite SDE \eqref{intro:SDE} and the classical formalism of Gibbs measures defined via specifications. Section \ref{se:wellposedness} proves the two well-posedness results, Theorems \ref{th:intro:wellposed-m=2} and \ref{th:intro:wellposed-m>2}. Finally, the examples from Sections \ref{se:intro:linear} and \ref{se:intro:repulsion} are developed in full detail in Sections \ref{se:linear-coeff} and \ref{se:repulsion}, respectively.

\section{Connection between local equations with original system}\label{se:mainthmproofs}

The objective of this section is to prove Theorems \ref{th:infSDE->local}, \ref{th:local->fixedpoint}, and \ref{th:fixedpoint->infSDE}  in Section \ref{subsec_SHMtolocal}, \ref{subsec_localtofix} and \ref{subsec_fixtoSHM} respectively. First, we quote two theorems of stochastic analysis that will feature in our arguments in this section as well as Section \ref{se:gibbsconnection}, the \emph{mimicking} and \emph{superposition} theorems.
The mimicking theorem is originally due to  \cite[Theorem 4.6]{gyongy1986mimicking}, and the version stated below follows from a generalization in \cite[Corollary 3.7]{brunick2013mimicking}. It is the basis for the proof of Theorem \ref{th:infSDE->local}.

\begin{theorem} \cite[Corollary 3.7]{brunick2013mimicking} \label{th:mimicking}
Let $\sigma > 0$.
Suppose $(\Omega,\F,\FF,\PP)$ is a filtered probability space supporting an $\FF$-Brownian motion $B$ of dimension $d$, an $\FF$-progressively measurable process $(b_t)_{t \ge 0}$ satisfying $\E\int_0^T|b_t|\,\d t < \infty$ for each $T > 0$, and an $\FF$-progressively measurable process $(X_t)_{t \ge 0}$ such that
\begin{align*}
X_t = X_0 + \int_0^t b_s\,\d s + \sigma W_t, \ \ t \ge 0.
\end{align*}
Let $\widehat{b} : \R_+ \times \R^d \to \R$ be any Borel measurable function satisfying
\begin{align*}
\widehat{b}(t,X_t) = \E[b_t\,|\,X_t], \ \ a.s. \ \text{ for a.e. } t \ge 0.
\end{align*}
Then there exists a weak solution $(\widehat{X}_t)_{t \ge 0}$ of the SDE
\begin{align*}
\d \widehat{X}_t = \widehat{b}(t,\widehat{X}_t)\,\d t + \sigma \d \widehat{W}_t,
\end{align*}
such that $\widehat{X}_t \stackrel{d}{=} X_t$ for each $t \ge 0$.
\end{theorem}

We next quote a superposition principle for SDEs. This originates from \cite{figalli2008existence} in the case of bounded coefficients, extended to unbounded cases in \cite{trevisan,bogachev2021ambrosio}. The $\R^\infty$ version we need is due to \cite{trevisan_diss}. We will use the superposition principle in the proof of Theorem \ref{th:fixedpoint->infSDE}, to construct a solution of the infinite SDE after first constructing a (weak) solution of the associated Fokker-Planck equation.

In the following, a \emph{cylindrical function} $g : \R^{\T_m} \to \R$ is one which depends on only finitely many coordinates. {We say $g$ is smooth and of compact support if it is a smooth and compactly supported function of finitely many variables. }
We say that a probability measure $\mu$ on $\R^{\T_m}$ \emph{solves the Fokker-Planck equation associated with the infinite SDE} if
\begin{equation}
\int_{\R^{\T_m}}\sum_{v\in \T_m}\partial_{vv} g(\bm x) \, \mu(\d \bm x) = \int_{\R^{\T_m}}\sum_{v\in \T_m}\Big(U'(x_v)+\sum_{u\in N(v)}K'(x_v-x_u)\Big)\partial_{v} g(\bm x)  \, \mu(\d \bm x). \label{FKP}
\end{equation}
hold for each  smooth cylindrical function $g:\R^{\T_m}\to \R$ of compact support; note that the sums in \eqref{FKP} involve only finitely many non-zero terms.

\begin{theorem} \cite[Theorem 7.1]{trevisan_diss} \label{th:superposition}
Suppose a probability measure $\mu$ on $\R^{\T_m}$ solves the Fokker-Planck equation associated with the infinite SDE. Suppose also that there exists $p > 1$ such that
\begin{align}
\int_{\R^{\T_m}} \big(|U'(x_v)|^p + |\K'(x_v-x_u)|^p\big)\,\mu(\d \bm x) < \infty, \label{asmp:superpos-integ}
\end{align}
for every edge $(v,u)$. Then there exists a stationary solution $\bm X$ of the infinite SDE such that $\bm X_t \sim \mu$ for each $t \ge 0$.
\end{theorem}

\subsection{From the infinite SDE to the local equation}\label{subsec_SHMtolocal}
We first prove Theorem \ref{th:infSDE->local}. Suppose $(X^v)_{v \in \T_m}$ is a SHM solution of the infinite SDE \eqref{intro:SDE}. Fix an edge $(u,v)$ in $\T_m$. Define 
\begin{align*}
b_1(x,y) &= \E\bigg[U'(X^u_t) + \sum_{w\in N(u) } \K'(X_t^u-X_t^w) \,\bigg|\, X_t^u=x,X_t^v=y\bigg] \\
	&= U'(x) +  \sum_{w\in N(u) } \E\big[\K'(X_t^u-X_t^w) \,\big|\, X_t^u=x,X_t^v=y\big],
\end{align*}
noting that the conditional expectations make sense thanks to \eqref{asmp:infSDE-integ},
and similarly
\begin{align*}
b_2(x,y) &= \E\bigg[U'(X^v_t) + \sum_{w\in N(v) } \K'(X_t^v-X_t^w) \,\bigg|\, X_t^u=x,X_t^v=y\bigg] \\
	&= U'(y) +  \sum_{w\in N(v) } \E\big[\K'(X_t^v-X_t^w) \,\big|\, X_t^u=x,X_t^v=y\big].
\end{align*}
We apply the mimicking Theorem \ref{th:mimicking} to find a weak solution $(X,Y)$ of the SDE system
\begin{equation}\label{Equ:tempmimick}
	\begin{cases}
	\d X_t = -b_1(X_t,Y_t) \, \d t + \sqrt{2} \, \d W_t,
	\\
	\d Y_t = -b_2(X_t,Y_t)  \,\d t  + \sqrt{2}  \,\d B_t,
	\end{cases}
\end{equation}
which satisfies $(X_t,Y_t)\stackrel{d}{=}(X_t^u,X_t^v)$ for each $t \ge 0$. 
Note that the homogeneity of $(X^v_t)_{v \in \T_m}$ implies the desired symmetry,
\begin{equation*}
(X_t,Y_t)\stackrel{d}{=}(X^u_t,X^v_t)\stackrel{d}{=}(X^{v}_t,X^u_t)\stackrel{d}{=}(Y_t,X_t),
\end{equation*}
and the time-stationarity $(X_t,Y_t) \stackrel{d}{=} (X_0,Y_0)$ follows from the same property of $(X^u,X^v)$.

For $w \in N(u) \setminus \{v\}$, we know $X^w_t$ and $X^v_t$ are conditionally independent given $X^u_t$, by the assumed Markov property. In particular,
\begin{align*}
\E\big[\K'(X_t^u-X_t^w) \,\big|\, X_t^u=x,X_t^v=y\big] &= \E\big[K'(X_t^u-X_t^w) \,\big|\, X_t^u=x\big] \\
	&= \E\big[\K'(X_t^u-X_t^v) \,\big|\, X_t^u=x\big] \\
	&= \E\big[\K'(X_t-Y_t) \,\big|\, X_t=x\big],
\end{align*}
with the second to last step using the assumed homogeneity to deduce $(X^u_t,X^w_t) \stackrel{d}{=} (X^u_t,X^w_t)$.
Therefore,
\begin{align*}
b_1(x,y) &= U'(x) + \sum_{w\in N(u) } \E\big[\K'(X_t^u-X_t^w) \,\big|\, X_t^u=x,X_t^v=y\big] 	\\
	&= U'(x) + \K'(x-y)+\sum_{w\in N(u)\setminus \{v\} }\E\big[\K'\big(X_t^u-X_t^w\big)|X_t^u=x,X_t^v=y\big] 	\\
	&= U'(x) + \K'(X_t-Y_t)+(m-1)\E\big[\K'(X_t-Y_t)\,|\,X_t=x\big].
\end{align*}
A similar argument shows that
\begin{align*}
b_2(x,y) &= U'(y) +  \K'(y-x) + \sum_{w\in N(v)\setminus \{u\} }\E\big[\K'\big(X_t^v-X_t^w\big)|X_t^u=x,X_t^v=y\big] 	\\
	&= U'(y) + \K'(y-x)+(m-1)\E\big[\K'(Y_t-X_t)\,|\,Y_t=y\big].
\end{align*}
Plug this into \eqref{Equ:tempmimick} to arrive at the desired form of the local equation \eqref{intro:localeq}. \hfill\qedsymbol

\subsection{From the local equation to the fixed point problem}\label{subsec_localtofix}
We next prove Theorem \ref{th:local->fixedpoint}.
Let $(X,Y)$ be a stationary symmetric solution of the local equation. Recall that the function $f(x) = \E[\K'(X_0-Y_0)\,|\,X_0=x]$ is locally Lebesgue integrable by assumption. Note also that \eqref{asmp:localeq-integ} and Jensen's inequality imply
\begin{align*}
\E[|f(X_0)|]=\E\Big[\Big|\E\big[\K'(X_0-Y_0)|X_0\big]\Big|\Big] \le \E\big[|\K'(X_0-Y_0)|\big]<\infty.
\end{align*}
Since $F(x)=\int_0^xf(u)\,\d u$, we can then write the SDE for $(X,Y)$ as
\begin{align*}
	\begin{cases}
	\d X_t = -\big(U'(X_t) + \K'(X_t-Y_t) + (m-1) F'(X_t)\big) \, \d t + \sqrt{2} \, \d W_t,
	\\
	\d Y_t = -\big(U'(Y_t) + \K'(Y_t-X_t) + (m-1) F'(Y_t)\big)  \,\d t  + \sqrt{2}  \,\d B_t.
	\end{cases}
\end{align*}
The drift of this two-dimensional process is the negative of the gradient of the function $(x,y) \mapsto U(x)+U(y)+K(x-y)+F(x)+F(y)$, and thus one expects the invariant measure to be
\begin{equation}\label{densityexpression}
\rho(x,y) = Z^{-1}\exp\big(-U(x)-U(y)-K(x-y)-(m-1)F(x)-(m-1)F(y)\big),
\end{equation}
where $Z$ is a normalizing constant making $\int \rho = 1$. To be precise, we may deduce \eqref{densityexpression}  along with the finiteness of $Z$ from \cite[Theorem 4.1.12]{bogachev2015fokker}, because of the assumption that $U'$, $\K'$, and $f$ are in $L^q_{\mathrm{loc}}(\R)$ for some $q > 2$, and because of the assumption \eqref{asmp:localeq-integ} which yields condition (i) therein. 
To prove that $F(x)$ solves the fixed point problem, note first that 
\begin{equation}\label{integ1}
\int_{\R^2} |\partial_x\rho(x,y)|\,\d x\d y \le \int_{\R^2} \big(|U'(x)|+|\K'(x-y)|+{(m-1)}|f(x)|\big)\rho(x,y)\,\d x\d y<\infty.
\end{equation}
By Fubini's theorem, we have
\begin{equation*}
\int_\R \rho(x,y) \,\d y=\int_\R\rho(0,y) \,\d y +\int_0^x \int_\R\partial_x\rho(z,y) \,\d y \d z. 
\end{equation*}
Therefore, using  \eqref{integ1}, we deduce that $x \mapsto \int_\R \rho(x,y)\d y$ is absolutely continuous on any compact set, and
\begin{equation*}
\frac{\d}{\d x}\int_\R \rho(x,y) \,\d y= \int_\R \partial_x\rho(x,y) \,\d y, \ \ a.e. \ x.
\end{equation*}
That is,
\begin{align*}
\frac{\d}{\d x}\bigg(&e^{-U(x)-(m-1)F(x)}\int_\R e^{-U(y)-\K(x-y)-(m-1)F(y)} \,\d y\bigg) 	\\
	=&\;\frac{\d}{\d x}\big(e^{-U(x)-(m-1)F(x)}\big)\int_\R e^{-U(y)-\K(x-y)-(m-1)F(y)} \,\d y \\
	&\; {-}e^{-U(x)-(m-1)F(x)}\int_\R K'(x-y)e^{-U(y)-\K(x-y)-(m-1)F(y)} \,\d y.
\end{align*}
By the product rule, this implies
\begin{equation*}
\frac{\d}{\d x}\bigg(\int_\R e^{-U(y)-\K(x-y)-(m-1)F(y)} \,\d y\bigg)= {-}\int_\R K'(x-y)e^{-U(y)-\K(x-y)-(m-1)F(y)} \,\d y,
\end{equation*}
which leads to
\begin{align*}
\dfrac{\d}{\d x}\log  \int_\R e^{-U(y)-\K(x-y)-(m-1)F(y)}\d y &= -\frac{\int_\R \K'(x-y) e^{-U(y)-K(x-y)-(m-1)F(y)}\d y}{\int_\R e^{-U(y)-\K(x-y)-(m-1)F(y)}\d y}	\\
		&=-\frac{\int_\R \K'(x-y)\rho(x,y)\d y}{\int_\R \rho(x,y)\d y} \\
		&=-\E\big[\K'(X_t-Y_t)|X_t=x\big]\\
		&=-f(x).
\end{align*}
Because $F'=f$ a.e., we can deduce that there exists a constant $C$ such that
\begin{equation}
C-\log\int_\R e^{-U(y)-K(x-y)-(m-1)F(y)}\,\d y=F(x), \ \ a.e. \label{pf:fp1}
\end{equation} 
Finally, the integrability condition required in the Definition \ref{def:fixedpoint} follows from \eqref{pf:fp1} and the fact that $\rho(x,y)$ from \eqref{densityexpression} is a probability density:
\begin{align*}
\int_\R e^{-U(x)-mF(x)}\d x &= e^{-C}\int_{\R^2}e^{-U(x)-U(y)-K(x-y)-(m-1)F(x)-(m-1)F(y)}\d x\d y<\infty.
\end{align*}
{\ } \vskip-1cm \hfill\qedsymbol

\subsection{From the fixed point problem to the infinite SDE}\label{subsec_fixtoSHM}

We lastly prove Theorem \ref{th:fixedpoint->infSDE}. Let $F$ be a solution of the fixed point problem with associated constant $C_1$, so that
\begin{align}
F(x) = C_1 -\log\int_\R e^{-U(y)-\K(x-y)-(m-1)F(y)}\d y, \ \ a.e.\ x \in \R. \label{pf:fp-Fdef}
\end{align}
First, note that $Z$ defined in \eqref{jointdensity-normalization} is finite because
\begin{align*}
e^{-C_1}\int_{\R^2}e^{-U(x)-U(y)-K(x-y)-(m-1)F(x)-(m-1)F(y)}\d x\d y &= \int_\R e^{-U(x)-mF(x)}\d x,
\end{align*}
with the latter being finite thanks to the assumption \eqref{fixedpoint-integrability}.
Thus, the symmetric probability density function $\rho(x,y)$ is well defined by \eqref{jointdensity-thm}.
Write
\begin{align*}
\rho_X(x) := \int_\R \rho(x,y)\,\d y
\end{align*}
for the marginal.
Before we go into the proof, let us introduce two functions $l(x)$ and $Q(x,y)$ by
\begin{align}\label{def:bdy}
&l(x)=e^{-\frac{U(x)}{m}-F(x)},\qquad Q(x,y)= e^{{-C_1} -\frac{U(x)}{m}-\frac{U(y)}{m}-K(x-y)}.
\end{align}
The following corollary connects $l$ with the marginal distribution of $\rho$ and writes the fixed point problem in terms of $l$ and $Q$.
This is essentially the notion of \emph{boundary law} for Gibbs measures on trees \cite[Chapter 12]{georgii2011gibbs}, on which we elaborate in Section \ref{se:boundarylaws}.

\begin{lemma}\label{Cor:lq}
With $l(x)$ and $Q(x,y)$ defined as in \eqref{def:bdy}, we have
\begin{equation}\label{landQ}
\int_\R Q(x,y)l(y)^{m-1}\,\d y=l(x).
\end{equation}
Moreover, the marginal distribution satisfies 
\begin{equation}\label{lboundary}
\rho_X(x) =C_2^{-1}l(x)^m, \quad \text{where} \quad C_2 := Z e^{-C_1},
\end{equation}
where $Z$ and $C_1$ are defined in \eqref{jointdensity-normalization} and \eqref{pf:fp-Fdef}, respectively.
\end{lemma}
\begin{proof}
These follow directly from \eqref{densityexpression} and \eqref{pf:fp-Fdef} by noting that
\begin{equation*}
\int_\R Q(x,y)l(y)^{m-1}\d y = e^{{-C_1}-\frac{U(x)}{m}}\int_\R e^{-U(y)-K(x-y)-(m-1)F(y)}\d y = e^{ -\frac{U(x)}{m}-F(x)}=l(x),
\end{equation*}
and also
\begin{align*}
\int_\R \rho(x,y)\d y &= Z^{-1} e^{-U(x)-(m-1)F(x)} \int_\R e^{-U(y)-K(x-y)-(m-1)F(y)}\d y \\
	&=Z^{-1}e^{C_1-U(x)-mF(x)}=Z^{-1}e^{ C_1}l(x)^m.
\end{align*}
\vskip-.5cm
\end{proof}

Now we aim to define a measure on $\R^{\T_m}$ with the help of $l(x)$ and $Q(x,y)$.
{In the following, let $\o \in \T_m$ denote an arbitrary root vertex, let $\T_m^k$ denote the closed ball of radius $k$ around $\o$ in $\T_m$, and let $\pi(v)$ denote the parent of a vertex $v$ in $\T_m \setminus \{\o\}$, defined as the unique neighbor of $v$ which belongs to the shortest path from $v$ to $\o$.}
We begin by inductively defining a measure $\mu_k$ on $\R^{\T_m^k}$ with  density function $p_k$. Note that $\T_m^0=\{\o\}$, so it makes sense to identify $\R^{\T_m^0} \cong \R$. Denote the boundary by $\partial \T_m^k := \T^k_m \setminus \T^{k-1}_m$ for $k \in \N$. Define
\begin{equation}\label{Def:p_k}
	\begin{aligned}
	&p_0(x)=C_2^{-1}l(x)^m,	\\
	&p_k\big(\bm x_{\T_m^k}\big)=C_2^{-1} \bigg(\prod_{v\in \partial \T_m^k}l(x_v)^{m-1}\bigg)\prod_{v \in \T_m^k \setminus\{\o\}} Q(x_{\pi(v)},x_v),\quad k\in \N.
	\end{aligned}
\end{equation}
Then, $p_0(x)$ defines a probability measure on $\R$ thanks to \eqref{lboundary}. Our next lemma shows that $p_k$ is indeed a density function on $\R^{\T_m^k}$ and it is consistent in the sense that the measure $\mu_n(d\bm{x})=p_n(\bm{x})\d \bm{x}$ restricted to $\T_m^k$ is equal to $\mu_k$ for each $k\leq n$:

\begin{lemma} \label{le:consistent}
For $k\in \N$ and $\bm{x}_{\T_m^{k-1}} \in \R^{\T_m^{k-1}}$, we have
\begin{equation*}
\int_{\R^{\partial \T^k_m}}p_k\big(\bm x_{\T_m^k}\big)\prod_{u\in \partial \T^k_m}\d x_u = p_{k-1}\big(\bm x_{\T_m^{k-1}}\big).
\end{equation*}
\end{lemma}
\begin{proof}
For $k=1$, the equation \eqref{landQ} yields
\begin{align*}
\int_{\R^{\partial\T_m^1}}\prod_{u\in\partial\T_m^1} Q(x_0,x_u)l(x_u)^{m-1}\d x_u= \prod_{u\in\partial\T_m^1}\int_\R Q(x_0,x)l(x)^{m-1}\d x=l(x_0)^m,
\end{align*}
because $|\partial\T_m^1|=m$.
For $k \ge 2$, note that $p_k$ factorizes as
\begin{align}
p_k\big(\bm x_{\T_m^k}\big) &= C_2^{-1}\bigg(\prod_{v \in \T_m^{k-1} \setminus\{\o\}} Q(x_{\pi(v)},x_v)\bigg)\prod_{v\in \partial \T_m^k}l(x_v)^{m-1} Q(x_{\pi(v)},x_v). \label{pf:p_k-rewrite}
\end{align}
Thus
\begin{align*}
\int_{\R^{\partial \T^k_m}}\! p_k\big(\bm x_{\T_m^k}\big)\d \bm x_{\partial \T_m^k} &= C_2^{-1}\bigg(\prod_{v \in \T_m^{k-1} \setminus\{\o\}} Q(x_{\pi(v)},x_v)\!\bigg) \int_{\R^{\partial \T_m^k}}\!\ \prod_{v\in \partial \T_m^k}l(x_v)^{m-1} Q(x_{\pi(v)},x_v)\,\d x_v \\
	&= C_2^{-1}\bigg(\prod_{v \in \T_m^{k-1} \setminus\{\o\}} Q(x_{\pi(v)},x_v)\! \bigg) \prod_{v\in \partial \T_m^k} \int_\R l(x)^{m-1} Q(x_{\pi(v)},x_v)\,\d x \\
	&= C_2^{-1}\bigg(\prod_{v \in \T_m^{k-1} \setminus\{\o\}} Q(x_{\pi(v)},x_v)\! \bigg) \prod_{v\in \partial \T_m^k} l(x_{\pi(v)}) \\
	&= C_2^{-1}\bigg(\prod_{v \in \T_m^{k-1} \setminus\{\o\}} Q(x_{\pi(v)},x_v)\! \bigg) \prod_{v\in \partial \T_m^{k-1}} l(x_v)^{m-1}.
\end{align*}
Indeed, the last step follows from the fact that $\pi(v) \in \partial \T_m^{k-1}$ whenever $v \in \partial \T_m^k$, and for any $u \in \partial \T_m^{k-1}$ there are $m-1$ choices of $v \in \partial \T_m^k$ such that $\pi(v)=u$. Recognizing the right-hand side as $p_{k-1}(\bm x_{\T_m^{k-1}})$, the proof is complete.
\end{proof}

Thanks to the consistency shown in Lemma \ref{le:consistent}, there exists by Daniell-Kolmogorov a unique probability measure $\mu$ on $\R^{\T_m}$ such that the marginal on $\R^{\T_m^k}$ equals $\mu_k$ for each $k$.
Our next lemma shows that $\mu$ solves the stationary Fokker-Planck equation corresponding to the infinite SDE, in the sense described just before Theorem \ref{th:superposition}.

\begin{proposition}\label{prop:verifydensity}
The probability measure $\mu$ satisfies the assumptions of Theorem \ref{th:superposition}.
\end{proposition}
\begin{proof}
We first show that \eqref{FKP} holds for every smooth cylindrical function $g$ of compact support.
We may write such a $g$ as $g(\bm{x}) = g_k(\bm{x}_{\T_m^k})$ for some smooth function $g_k : \R^{\T_m^k} \to \R$ of compact support. Noting that the neighbors of vertices in $\T_m^k$ are contained in $\T_m^{k+1}$, we must show
\begin{equation*}
\begin{aligned}
 \int_{\R^{\T_m^{k+1}}} &\sum_{v\in \T_m^k}\partial_{vv} g_k\big(\bm x_{\T_m^{k}}\big)\,\mu_{k+1}\big(\d \bm x_{\T_m^{k+1}}\big) 
 \\
 &=\int_{\R^{\T_m^{k+1}}}\sum_{v\in \T_m^k}\Big(U'(x_v)+\sum_{u\in N(v)}K'(x_v-x_u)\Big)\partial_{v} g_k\big(\bm x_{\T_m^{k}}\big) \,\mu_{k+1}\big(\d \bm x_{\T_m^{k+1}}\big) .
 \end{aligned}
\end{equation*}
Recall that $p_{k+1}$ is the density of $\mu_{k+1}$. Fix $v \in \T_m^k$ and integrate by parts to get
\begin{align*}
 \int_{\R^{\T_m^{k+1}}} \partial_{vv} g_k\big(\bm x_{\T_m^{k}}\big) \,\mu_{k+1}\big(\d \bm x_{\T_m^{k+1}}\big) &= -\int_{\R^{\T_m^{k+1}}} \partial_v g_k\big(\bm x_{\T_m^{k}}\big) \partial_v  \log p_{k+1}\big(\bm x_{\T_m^{k+1}}\big) \,\mu_{k+1}\big(\d \bm x_{\T_m^{k+1}}\big).
\end{align*}
Note that this integration by parts is justified by the standing assumption, which ensures that that $e^{-U}$ and $e^{-\K}$ are locally absolutely continuous as well as strictly positive a.e.
{Note that $Q$ is symmetric in its two variables, which implies $\partial_1 Q(x,y)=\partial_2Q(y,x)$, if $\partial_i$ denotes the derivative in the $i^\text{th}$ variable.}
By definition of $p_{k+1}$ in \eqref{Def:p_k},  we have
\begin{align*}
\partial_{v}\log p_{k+1}\big(\bm x_{\T_m^{k+1}}\big)&=\sum_{u\in N(v)}\frac{\partial_1 Q(x_v,x_u)}{Q(x_v,x_u)}   \\
	&= \sum_{u\in N(v)}\Big(-\frac{U'(x_v)}{m}-K'(x_v-x_u)\Big)  \\
	&= -U'(x_v)-\sum_{u\in N(v)}K'(x_v-x_u). 
\end{align*}
Therefore, 
\begin{equation*}
	\begin{aligned}
	\int_{\R^{\T_m^{k+1}}} \partial_{vv} g_k&\big(\bm x_{\T_m^{k}}\big)\,\mu_{k+1}\big(\d \bm x_{\T_m^{k+1}}\big)
  \\
  &=\int_{\R^{\T_m^{k+1}}} \Big(U'(X_v)+\sum_{u\in N(v)}K'(x_v-x_u)\Big)\partial_{v} g_k(\bm x_{\T_m^{k}}\big)\,\mu_{k+1}\big(\d \bm x_{\T_m^{k+1}}\big).
	\end{aligned}
\end{equation*}
Sum over $v\in \T_m^{k}$ to deduce the Fokker-Planck equation.

It remains to check the claimed integrability condition \eqref{asmp:superpos-integ}.
Let $\{u,v\}$ be any edge. Since the image of $\mu$ by $x \mapsto (x_u,x_v)$ equals $\rho(x,y)\d x \d y$,  recalling the form of $\rho$ from \eqref{jointdensity-thm}, we see that \eqref{asmp:superpos-integ} follows immediately from \eqref{asmp:tech1'}.
\end{proof}

Thanks to Proposition \ref{prop:verifydensity} and Theorem \ref{th:superposition}, there exists a stationary solution $\bm X$ of the infinite SDE, with $\bm X_t \sim \mu$ for each $t \ge 0$. Note that the required integrability \eqref{asmp:infSDE-integ} holds because Proposition \ref{prop:verifydensity} showed that \eqref{asmp:superpos-integ} holds. The next lemma will check the Markov property and homogeneity, showing that $\bm{X}$ is in fact an SHM solution.

\begin{lemma} \label{le:Markov-autinv}
The stationary solution $\bm X$ of the infinite SDE constructed above is in fact an SHM solution.
\end{lemma}
\begin{proof}
The Markov property follows quickly from the form of the marginal density $p_k$ of $\bm X^{\T_m^k}_0$. Indeed, the formula \eqref{Def:p_k} exhibits a factorization of $p_k$ over edges of $\T_m^k$, which easily implies that $\bm X^{\T_m^k}_0$ is a Markov random field, in the sense that the second bullet point of Definition \ref{def:infiniteSDE} holds true when $S$, $A$, and $B$ are contained in $\T_m^k$. 

We next prove homogeneity. Recall that $\o$ denoted an arbitrarily chosen ``root" vertex of $\T_m$, with $\T_m^k$ denoting the ball of radius $k$ around $\o$ in $\T_m$. It suffices to show that
\begin{align}
(X^v_0)_{v \in \T_m^k} \stackrel{d}{=} (X^{\varphi(v)}_0)_{v \in \T_m^k} \label{pf:automorphism}
\end{align}
for every $k \in \N$ and every automorphism $\varphi$ of $\T_m$. First, suppose $\varphi$ fixes the root, $\varphi(\o)=\o$. Then $\varphi|_{\T_m^k}$ is an automorphism of the induced subgraph $\T_m^k$. By definition of $p_k$ in \eqref{Def:p_k}, it is clear that
\begin{align*}
p_k((x_v)_{v \in \T_m^k}) = p_k((x_{\varphi(v)})_{v \in \T_m^k}),
\end{align*}
and we deduce \eqref{pf:automorphism} in this case. With this first case out of the way, it suffices to prove that \eqref{pf:automorphism} holds just for those automorphisms which map $\o$ to a neighbor of $\o$, which we will call \emph{shifts}. Indeed, a general automorphism can be achieved as a composition of shifts and perhaps an automorphism which fixes the root. With this in mind, let $\widetilde{\o}$ be a neighbor of $\o$, and let $\varphi$ be an automorphism with $\varphi(\o)=\widetilde{\o}$. Let $\widetilde\T_m^k$ be the ball of radius $k$ with respect to root $\widetilde{\o}$. It suffices to show that the density function of $(X^v_0)_{v \in \widetilde\T_m^k}$ is also $p_k$ for all $k\in\N$. Note that $\widetilde\T_m^k\subset \T_m^{k+1}$  and $\T_m^{k+1} \setminus \widetilde\T_m^k = \partial \T_m^{k+1} \setminus \partial \widetilde\T_m^k$. We may then compute the density of $(X^v_0)_{v \in \widetilde\T_m^k}$ by
\begin{align*}
\widetilde{p}_k\big(\bm x_{\widetilde\T_m^k}\big) &:= \int_{\R^{\T_m^{k+1} \setminus \widetilde\T_m^k}} p_{k+1} \big(\bm x_{\T_m^{k+1}}\big) \d  \bm x_{\T_m^{k+1} \setminus \widetilde\T_m^k} 	\\
	&= C_2^{-1}\int_{\R^{\T_m^{k+1} \setminus \widetilde\T_m^k}} \bigg(\prod_{v\in \partial \T_m^{k+1}}l(x_v)^{m-1}\bigg)\prod_{v \in \T_m^{k+1} \setminus\{\o\}} Q(x_{\pi(v)},x_v)\, \d \bm x_{\T_m^{k+1} \setminus \widetilde\T_m^k}  \\
	&= C_2^{-1} \left(\prod_{v \in \partial \widetilde\T_m^k \cap \partial \T_m^{k+1}}  l(x_v)^{m-1} \right)\prod_{v \in \widetilde\T_m^{k} \setminus\{\o\}} Q(x_{\pi(v)},x_v) \\
	&\qquad \cdot \int_{\R^{\T_m^{k+1} \setminus \widetilde\T_m^k}} \prod_{v \in \partial \T_m^{k+1} \setminus \partial \widetilde\T_m^k} l(x_v)^{m-1} Q(x_{\pi(v)},x_v) \, \d \bm x_{\T_m^{k+1} \setminus \widetilde\T_m^k},
\end{align*}
where $\pi(v)$ still denotes the parent of vertex $v$ when the root still $\o$. Factorizing and applying \eqref{landQ}, the remaining integral expression can be written as
\begin{align}
\prod_{v \in \partial \T_m^{k+1} \setminus \partial \widetilde\T_m^k}\int_\R  l(x_v)^{m-1} Q(x_{\pi(v)},x_v) \, \d x_v = \prod_{v \in \partial \T_m^{k+1} \setminus \partial \widetilde\T_m^k}l(x_{\pi(v)}). \label{pf:aut1}
\end{align}
For each $v \in \partial \T_m^{k+1} \setminus \partial \widetilde\T_m^k$, the the cardinality of the set $S_v:=\{u \in \partial \T_m^{k+1} \setminus \partial \widetilde\T_m^k : \pi(u)=\pi(v)\}$ is exactly $m-1$; that is, the factor $l(x_{\pi(v)})$ appears $m-1$ times. Moreover, the set $\{\pi(v) : v \in \partial\T_m^{k+1} \setminus \partial \widetilde\T_m^k\}$ is exactly equal to $\partial \widetilde\T_m^k \setminus \partial\T_m^{k+1}$. Hence, the right-hand side of \eqref{pf:aut1} equals
\begin{align*}
\prod_{v \in \partial \widetilde\T_m^k \setminus \partial\T_m^{k+1}}l(x_v)^{m-1}.
\end{align*}
Putting it together, we find
\begin{align*}
\widetilde{p}_k\big(\bm x_{\widetilde\T_m^k}\big) = C_2^{-1} \left(\prod_{v \in \partial \widetilde\T_m^k}  l(x_v)^{m-1} \right)\prod_{v \in \widetilde\T_m^{k} \setminus\{\o\}} Q(x_{\pi(v)},x_v) = p_k\big(\bm x_{\widetilde\T_m^k}\big),
\end{align*}
which completes the proof.
\end{proof}

Lemma \ref{le:Markov-autinv} completes the desired construction of a SHM solution from the given solution $F$ of the fixed point problem. 
We finally prove the last claim of the theorem, namely that $F'$ is a version of $\E[\K'(X-Y)\,|\,X=x]$ when $(X,Y) \sim \rho$. Since $F > 0$, it is equivalent to show that $(e^{-F})'(x)=-\E[\K'(X-Y)\,|\,X=x] e^{-F(x)}$. Note by \eqref{pf:fp-Fdef} that
\begin{equation}
e^{-F(x)} = e^{-C_1} \int_\R e^{-U(y)-\K(x-y)-(m-1)F(y)} \d y = \int_\R g(x,y) \d y, \label{pf:fp->inf-Ffinal} 
\end{equation}
where we define $g(x,y) =  e^{-U(y)-\K(x-y)-(m-1)F(y)}$.
For any compact set $A \subset \R$, H\"older's inequality yields
\begin{align*}
&\int_A \int_\R  |\partial _x g(x,y)| \d y \d x 
\\
& \ \ \leq \int_A \bigg(\int_\R |K'(x-y)|^m e^{-U(y)-mK(x-y)}\d y\bigg)^{1/m} \bigg(\int_\R e^{-U(y)-mF(y)}\d y\bigg)^{(m-1)/m} \d x,
\end{align*}
which is finite by assumptions \eqref{asmp:tech2} and \eqref{fixedpoint-integrability}.
That is, the function $x\mapsto  \int_\R \partial_x g(x,y) \d y$ belongs to $L^1_{\mathrm{loc}}(\R)$, and it is the derivative of $e^{-F}$ in the sense of distributions. We deduce that $e^{-F}$ is absolutely continuous, with
\begin{align*}
(e^{-F})'(x) &= -  e^{-C_1} \int_\R K'(x-y) e^{-U(y)-\K(x-y)-(m-1)F(y)} \d y,
\end{align*}
for a.e.\ $x \in \R$. Hence, from \eqref{pf:fp->inf-Ffinal} and the definition of $\rho$ from \eqref{densityexpression}, we get
\begin{align*}
F'(x)=-\frac{(e^{-F})'(x)}{e^{-F(x)}} &= \frac{\int_\R K'(x-y) e^{-U(y)-\K(x-y)-(m-1)F(y)}\,\d y}{\int_\R e^{-U(y)-\K(x-y)-(m-1)F(y)} \,\d y} \\
	&= \frac{\int_\R K'(x-y) \rho(x,y) \,  \d y}{\int_\R\rho(x,y) \,  \d y},
\end{align*}
for  a.e.\ $x \in \R$, which completes the proof. \hfill \qedsymbol

\section{Connection with Gibbs measure} \label{se:gibbsconnection}

In this section, we explain how the  infinite SDE   \eqref{intro:SDE} relates to the classical concept of a (infinite volume) Gibbs measure associated with a specification on $\T_m$.

We define admissible $(U,W)$ and a Gibbs specification $\gamma$ for $(U,W)$ as follows.
Let $\TT$ be the set of all finite subset of $\T_m$. 
For $T \in \TT$, let $E_T$ denote the set of unordered edges of $\T_m$ for which at least one vertex belongs to $T$. We call $(U,W)$ admissible if 
\begin{equation*}
  Z_T(\bm x_{T^c})=\int_{\R^T}\exp\Big(-\sum_{v\in T}U(x_v)- \sum_{\{u,v\} \in E_T}K(x_u-x_v)\Big) \,\d \bm x_T<\infty
\end{equation*}
for all $T \in \TT$, $\bm x_{T^c}\in\R^{T^c}$. We then define the probability measure $\gamma_T(\cdot\,|\,\bm x_{T^c})$ on $\R^{T}$ by
\begin{equation*}
	\begin{aligned}
	\gamma_T(A\,|\,\bm x_{T^c})=\frac{1}{Z_T(\bm x_{T^c})}\int_A & \exp\Big(-\sum_{v\in T}U(x_v)-  \sum_{\{u,v\} \in E_T}K(x_u-x_v)\Big)\,\d \bm x_T,
	\end{aligned}
\end{equation*}
for Borel sets $A \subset \R^{T}$. Note that $\gamma_T(A\,|\,\bm x_{T^c})$ and $Z_T(\bm x_{T^c})$ depend on $\bm x_{T^c}$ only through $\bm x_{\partial T}$, where $\partial T$ is the set of vertices in $T^c$ which are adjacent to some vertex in $T$.
The set of Gibbs measures associated with the specification $\gamma$ is defined by
\begin{align*}
	\GG(\gamma) =\big\{\mu \in \P(\R^{\T_m}) : \mu(\cdot\,|\,\bm x_{T^c})=\gamma_T(\cdot\,|\, \bm x_{T^c}), \;\mu\text{-a.s.}, \ \forall T\in\TT\big\}.
\end{align*}
That is, $\mu \in \GG(\gamma)$ if and only if $\gamma_T(\cdot\,|\,\bm x_{T^c})$ is a version of the conditional law of $\bm x_T$ given $\bm x_{T^c}$ under $\mu$, for each $T \in \TT$.

\begin{theorem} \label{th:gibbsconnection}
Assume $(U,W)$ is admissible.
If $\bm X=(X^v)_{v \in \T_m}$ is a SHM solution of the infinite SDE (in the sense of Definition \ref{def:infiniteSDE}), and if
\begin{align}
f(x) = \E[\K'(X^v_0-X^u_0)\,|\, X^v_0=x] \label{gibbsthm-def-f}
\end{align}
as well as $U'$ and $\K'$ are locally bounded, then
the law of $\bm X_0$ belongs to $\GG(\gamma)$. Conversely, if $\mu \in \GG(\gamma)$ satisfies
\begin{align}
\int_{\R^{\T_m}}(|U'(x_v)|^p + |\K'(x_v-x_u)|^p) \,\mu(d\bm x) < \infty, \label{asmp:Gibbs-int-p}
\end{align}
for every edge $(u,v)$ in $\T_m$ for some $p > 1$, then $\mu$ is the law of a stationary solution of the infinite SDE.
\end{theorem}
\begin{proof}
First, suppose we are given a SHM solution of the infinite SDE with law $\mu$. For adjacent vertices $v,u \in \T_m$, define $f$ by \eqref{gibbsthm-def-f};
the homogeneity and Markov properties ensure that a version can be defined which does not depend on the choice of $(v,u)$.
Let $T\in \TT$. For $v \in T$ and $u \sim v$, the homogeneity and Markov properties imply that
\begin{align*}
\E[\K'(X^v_t-X^u_t)\,|\,X^T_t] = \begin{cases}
\K'(X^v_t-X^u_t) &\text{if } u \in T \\
f(X^v_t) &\text{if } u \notin T.
\end{cases}
\end{align*}
By the mimicking Theorem \ref{th:mimicking}, there exists a weak solution $\bm Y^T:=(Y^v)_{v\in T}$ of the SDE
\begin{equation*}
	\d Y_t^v=-\Big(U'(Y_t^v)+|N(v)\setminus T| \, f(Y_t^v)+\!\!\sum_{u\in N(v)\cap T}K'(Y_t^v-Y_t^u)\Big)\d t+\sqrt 2 \, \d \widehat W_t^v, \quad v \in T,
\end{equation*}
which satisfies $\bm X_t^{T} \overset{d}{=} \bm Y_t^{T}$ for all $t \ge 0$.
We may apply \cite[Theorem 4.1.12]{bogachev2015fokker}, thanks to the assumption that $U'$, $\K'$, and $f$ are locally bounded {along with the integrability assumption \eqref{asmp:infSDE-integ} (which yields condition (i) in \cite[Theorem 4.1.12]{bogachev2015fokker}),} to deduce that the density function $p_T$ of the marginal distribution on any finite subset $T\subset \T_m$ is 
\begin{equation*}
	p_T(\bm x_T)=C_T  \exp\Big(-\sum_{v\in T}\big(U(x_v)+|N(v)\setminus T|\, F(x_v)\big)-  \sum_{\{u,v\} \in E_T^\circ }K(x_v-x_u)\Big),
\end{equation*}
where $F$ is an antiderivative of $f$, $C_T$ is a normalizing constant, and $E_T^\circ$ is the set of unordered edges for which both vertices belong to $T$. Now consider two finite subsets $S,T\subset \T_m$ such that $S\cup \partial S \subset T$. The same argument with $T \setminus S$ in place of $T$ shows that
\begin{align*}
p_{T\setminus S}(\bm x_{T\setminus S})=C_{T\setminus S} \exp\Big(\!-\!\!\sum_{v\in T\setminus S}\big(U(x_v)&+|N(v)\setminus (T\setminus S)|\, F(x_v)\big) - \!\!\!\!\! \sum_{\{u,v\} \in E_{T\setminus S}^\circ }K(x_v-x_u)\Big).
\end{align*}
The conditional density of $\bm X^S_t$ given $\bm X^{T \setminus S}_t$ is then
\begin{align*}
\frac{p_T(\bm x_T)}{p_{T\setminus S}(\bm x_{T\setminus S})} = \frac{C_T}{C_{T \setminus S}}\exp\Bigg( &  - \sum_{v \in T}|N(v)\setminus T|\, F(x_v) + \sum_{v \in T \setminus S}|N(v)\setminus (T\setminus S)|\, F(x_v) \\
	& -\sum_{v \in S} U(x_v) - \!\!\sum_{\{u,v\} \in E_T^\circ \setminus E_{T\setminus S}^\circ} K(x_v-x_u) \Bigg).
\end{align*}
We simplify this via three observations. First, since $S \cup \partial S  \subset T$, it holds that $N(v) \setminus T = \emptyset$ for each $v \in S$, so the first sum can be restricted to $v \in T \setminus S$. Second, since $S$ and $T^c$ are disjoint, we have
\begin{align*}
|N(v)\setminus T| - |N(v)\setminus (T\setminus S)| &= |N(v)\cap T^c| - |(N(v)\cap T^c)\cup (N(v)\cap S)| \\
	&= - |N(v)\cap S|.
\end{align*}
Lastly, $E_T^\circ \setminus E_{T\setminus S}^\circ$ is the set of (unordered) edges for which both vertices belong to $T$ and at least one belongs to $S$; this is exactly the same as $E_S$ because $S \cup \partial S  \subset T$. Combining these three observations, we find
\begin{align*}
\frac{p_T(\bm x_T)}{p_{T\setminus S}(\bm x_{T\setminus S})} = \frac{C_T}{C_{T \setminus S}}\exp\Bigg( &   \sum_{v \in T \setminus S}|N(v)\setminus S|\, F(x_v) -\sum_{v \in S} U(x_v) - \!\sum_{\{u,v\} \in E_S} K(x_v-x_u) \Bigg).
\end{align*}
For $v \in T \setminus S$, we either have $|N(v) \setminus S|=0$ if $v \notin \partial S $, or $|N(v) \setminus S|=1$ if $v \in \partial S $. Thus
\begin{align}
\frac{p_T(\bm x_T)}{p_{T\setminus S}(\bm x_{T\setminus S})} = \frac{C_T}{C_{T \setminus S}}\exp\Bigg( &   \sum_{v \in \partial S }  F(x_v) -\sum_{v \in S} U(x_v) -  \!\sum_{\{u,v\} \in E_S} K(x_v-x_u) \Bigg). \label{pf:conddens1}
\end{align}
This depends on $\bm x_{T \setminus S}$ only through $\bm x_{\partial S }$. Similarly, the quantity $Z_S(\bm x_{S^c})$ depends on $\bm x_{S^c}$ only through $\bm x_{\partial S }$, and we see that
\begin{align*}
{Z_S(\bm x_{S^c})} =  \frac{C_{T\setminus S}}{C_T}  \exp\bigg(-\sum_{v\in\partial S}F(x_v)\bigg)<\infty,
\end{align*}
for all $S\in \TT$ and $\bm x_{S^c}\in \R^{S^c}$. Because the conditional density of $\bm X^S_t$ given $\bm X^{T \setminus S}_t$ does not depend on the choice of $T \supset S \cup \partial S $ and depends only on $\bm X^{\partial S }_t$, it must be the same as the conditional density of $\bm X^S_t$ given $\bm X^{S^c}_t$. This density, as \eqref{pf:conddens1} shows, is exactly the same as that of $\gamma_T(\cdot\,|\,\bm x_{S^c})$.

Conversely, fix $\mu \in \GG(\gamma)$. For $k \in \N$, let $\T_m^k$ again denote the closed ball of radius $k$ around an (arbitrarily) fixed root vertex in $\T_m$. 
As in the proof of Proposition \ref{prop:verifydensity}, we will show that, for any smooth cylindrical function  $g : \R^{\T_m} \to \R$ of compact support,
\begin{align}\label{pfgibs:equa1}
\int_{\R^{\T_m}} \sum_{v\in \T_m}\partial_{vv} g(\bm x) \,\mu(\d \bm x) = \int_{\R^{\T_m}} \sum_{v\in \T_m}\Big(U'(x_v)+\sum_{u\in N(v)}K'(x_v-x_u)\Big)\partial_{v} g(\bm x)  \,\mu(\d \bm x).
\end{align}
Such a function $g$, by definition, may be represented as $g(\bm x)=g_k(\bm x_{\T_m^k})$ for a smooth function $g_k : \R^{\T_m^k} \to \R$ of compact support. Once this is established, the existence of a stationary solution of the infinite SDE with law $\mu$ follows from the superposition principle (in the form of Theorem \ref{th:superposition}) and the integrability assumption \eqref{asmp:Gibbs-int-p}.

The identity \eqref{pfgibs:equa1} will essentially follow from the law of total expectation and integration by parts. By definition of $\GG(\gamma)$, the measure $\mu(\cdot\,|\,\bm x_{(\T_m^k)^c})$ on $\R^{\T_m^k}$ has density
\begin{align*}
	p(\bm x_{\T_m^k} \,|\, \bm x_{\partial \T_m^k}) &=\frac{1}{Z_{\T_m^k}(\bm x_{\partial \T_m^k})}\exp\bigg(-\sum_{v\in \T_m^k}U(x_v)-\!\sum_{\{u,v\}\in E_{\T_m^k}} K(x_v-x_u)\bigg).
\end{align*}
Integrating by parts, we have for each $v \in \T_m^k$ and $\bm x_{\partial \T_m^k} \in \R^{(\T_m^k)^c}$ that
\begin{align*}
\int_{\R^{\T_m^k}} &\partial_{vv} g_k(\bm x_{\T_m^k}) p(\bm x_{\T_m^k} \,|\, \bm x_{\partial \T_m^k})\,\d\bm x_{\T_m^k} \\
	&= \int_{\R^{\T_m^k}}\partial_v g_k(\bm x_{\T_m^k}) \Big( U'(x_v) + \sum_{u \in N(v)}\K'(x_v-x_u) \Big) p(\bm x_{\T_m^k} \,|\, \bm x_{\partial \T_m^k})\,\d\bm x_{\T_m^k}.
\end{align*}
Summing over $v \in \T_m^k$ and integrating with respect to the marginal law of $\mu$ on $\R^{(\T_m^k)^c}$ yields \eqref{pfgibs:equa1}.
\end{proof}

\begin{remark}
It is clear from the proof that the first part of Theorem \ref{th:gibbsconnection} remains true if $(X^v)_{v \in \T_m}$ is merely required to be a stationary solution of the infinite SDE satisfying the Markov property. Automorphism invariance is not needed.
\end{remark}

\subsection{Boundary laws} \label{se:boundarylaws}

Our fixed point problem in Definition \ref{def:fixedpoint} can be viewed as a reformulation of the notion of \emph{boundary law} for Gibbs measures on trees, which is described in detail in \cite[Chapter 12]{georgii2011gibbs} in the case of a finite state space. To see the connection, let $S \subset \R$ be finite. Consider the matrix $Q\in \R_+^{S \times S}$ defined by
\begin{equation*}
	Q_{x,y}=e^{-U(x)-U(y)-K(x-y)}.
\end{equation*} 
As in \cite[Chapter 12, Definition 12.10]{georgii2011gibbs}, let us define a \emph{boundary law for $Q$} to be a family of edge-indexed vectors $\{l^{u,v}: v \in \T_m, \, u\in N(v)\} \subset \R_+^S$ satisfying
\begin{equation*}
	l^{u,v}_x=c_{u,v}\prod_{w\in N(u)\setminus\{v\}}\big(Ql^{w,u}\big)_x
\end{equation*}
for some constants $c_{u,v} > 0$. In the homogeneous case, where $l^{u,v}=l$ does not depend on $(u,v)$, the above equation becomes $l_x=c((Ql)_x)^{m-1}$. In the finite state space case, boundary laws are known to characterize Gibbs measures having the Markov chain property; see \cite[Theorem 12.12]{georgii2011gibbs}, or \cite[Corollary 12.17]{georgii2011gibbs} for the homogeneous case.
The functions we called $Q$ and $l$ in Section \ref{subsec_fixtoSHM} can be seen as infinite-dimensional counterparts of the matrix/vector pair $(Q,l)$ described in this paragraph, and Lemma \ref{Cor:lq} gives the natural counterpart to $l_x=c((Ql)_x)^{m-1}$ in our setting of continuous state space.

\section{Existence and uniqueness} \label{se:wellposedness}

\subsection{Existence and uniqueness, $m=2$} \label{se:wellposedness:m=2}
This section proves Theorem \ref{th:intro:wellposed-m=2}. We treat existence and uniqueness separately.

\vskip.2cm

\noindent\textbf{Uniqueness.}  Let $F^1 (x)$, $F^2 (x)$ be two solutions of the fixed point problem.
For $k=1,2$,  the statement that $F^k$ solves the fixed point problem means that $F^k(x)$ is a scalar shift of the function
\begin{align*}
x \mapsto -\log\int_\R e^{-U(y)-\K(x-y)-F^k(y)}\,\d y.
\end{align*}
Define the corresponding joint density function $\rho^k$ on $\R^2$ by
\begin{equation*}
	\rho^k (x,y)=C^k e^{-U(x)-U(y)-K(x-y)-F^k (x)-F^k (y)}, 
\end{equation*}
with $0 < C^k < \infty$a normalizing constant. Note that the fixed point equation implies that $C^k$ equals a positive constant times
\begin{align*}
\Big(\int_{\R^2} e^{-U(x)-2F^k (x)}\d x\Big)^{-1},
\end{align*}
which is finite thanks to \eqref{fixedpoint-integrability}. 
By Theorem \ref{th:fixedpoint->infSDE} and \eqref{asmp:tech2}, $F^k$ is absolutely continuous, and its derivative $f^k$ satisfies
\begin{align}
f^k(x) &= \frac{\int_\R \K'(x-y)\rho^k(x,y)\,\d y}{\int_\R  \rho^k(x,y)\,\d y}. \label{pf:m=2,1}
\end{align}
Next, note that $\rho^k(x,y)=\rho^k(y,x)$, and thus the marginal densities  are equal; we define
\begin{align*}
\rho^k_X(x) = \int_\R \rho^k(x,y) \,\d y = \int_\R \rho^k(y,x) \, \d y.
\end{align*}
Since $F^k$ is a solution of the fixed point problem, this marginal distribution takes the form
\begin{align*}
	\rho^k _X(x)=C_X^k e^{-U(x)-2F^k (x)},
\end{align*}
for a normalizing constant $0 < C_X^k < \infty$.

Let $H$ denote the relative entropy, defined for two positive probability densities $\nu$ and $\mu$ on $\R^k$ by
\begin{align*}
H(\nu\,|\,\mu) := \int_{\R^k} \nu(x) \log \frac{\nu(x)}{\mu(x)} \, \d x.
\end{align*}
We compute
\begin{align*}
H(\rho^1 \,|\,\rho^2 ) &= \int_{\R^2} \big(F^2 (x)+F^2 (y)-F^1 (x)-F^1 (y)\big)\rho^1 (x,y)\,\d x\,\d y+\log\frac{C^1 }{C^2 } \\
	&= 2\int_{\R} \big(F^2 (x) -F^1 (x)\big)\rho^1 _X(x)\,\d x +\log\frac{C^1 }{C^2 }.
\end{align*}
Similarly,
\begin{align*}
H(\rho^2 \,|\,\rho^1 ) &=  2\int_{\R} \big(F^1 (x) -F^2 (x)\big)\rho^2_X(x)\,\d x +\log\frac{C^2}{C^1}.
\end{align*}
For the marginal distribution, we compute similarly
\begin{align*}
H(\rho^1 _X\,|\,\rho^2 _X) &= 2\int \big(F^2 (x)-F^1 (x)\big)\rho^1_X(x)\,\d x +\log\frac{C_X^1 }{C_X^2 }, \\
H(\rho^2 _X\,|\,\rho^1 _X) &= 2\int \big(F^1 (x)-F^2 (x)\big)\rho^2_X(x)\,\d x+\log\frac{C_X^2 }{C_X^1 },
\end{align*}
Hence,
\begin{align}
\begin{split} H(\rho^1 \,|\,\rho^2 )+ H(\rho^2 \,|\,\rho^1 ) &= 2\int_{\R} \big(F^2 (x) -F^1 (x)\big)(\rho^1 _X(x)-\rho^2 _X(x))\,\d x \\
 	&= H(\rho^1 _X\,|\,\rho^2 _X)+H(\rho^2 _X\,|\,\rho^1 _X) . \end{split} \label{pf:m=2,2}
\end{align}
On the other hand, it follows from the chain rule for relative entropy that $H(\rho^1 \,|\,\rho^2 ) \ge  H(\rho^1_X \,|\,\rho^2_X )$ and $H(\rho^2 \,|\,\rho^1 ) \ge  H(\rho^2_X \,|\,\rho^1_X )$, with equality if and only if $\rho^1(x,y)/\rho^1_X(x) = \rho^2(x,y)/\rho^2_X(x)$ for almost every $x,y$. Hence, and \eqref{pf:m=2,2} and \eqref{pf:m=2,1} imply $f^1=f^2$ a.e., so $F^1-F^2$ is constant, and the claimed uniqueness of the fixed point problem follows. 

Under the additional assumption that $U',\K'\in L^q_{\mathrm{loc}}(\R)$ for some $q > 2$, the claimed uniqueness of the stationary symmetric solution of the local equation  follows from Theorem \ref{th:local->fixedpoint}, and then the uniqueness of the SHM solution of the infinite SDE follows from Theorem \ref{th:infSDE->local}.

\vskip.2cm

\noindent\textbf{Existence.} 
Define the positive measure $\mu(\d x) = e^{-U(x)}\d x$.
Define the linear operator
\begin{align*}
S\varphi(x) := \int_\R  \varphi(y) e^{-\K(x-y)-U(y)}\,\d y.
\end{align*}
Under the additional assumption \eqref{asmp:m=2}, $S$ is a Hilbert-Schmidt (and thus compact) operator on the Hilbert space $L^2(\mu)$ with positive spectral radius; see \cite[Theorem VI.23]{reed1978methods}.
Note that $S$ is positive, in the sense that $S\varphi \ge 0$ whenever $\varphi \ge 0$.
By the Krein-Rutman theorem {(see \cite[Theorem A]{bonsall1958linear} for a concise English reference)}, there exists an nonzero eigenfunction $\varphi \ge 0$ with positive eigenvalue $\lambda > 0$. Namely,
\begin{align*}
\varphi(x) = \lambda S\varphi(x) = \lambda \int_\R  \varphi(y) e^{-\K(x-y)-U(y)}\,\d y, \ \ a.e.\ x.
\end{align*}
It follows from this identity, and the fact that $\varphi$ is not identically zero, that in fact $\varphi > 0$ a.e. Then $F:=-\log \varphi$ satisfies the fixed point equation \eqref{fixedpoint} with $C=-\log\lambda$.

Under the additional assumption \eqref{asmp:tech1}, the existence of a SHM solution of the infinite SDE follows from Theorem \ref{th:fixedpoint->infSDE}, and the existence of a stationary symmetric solution of the local equation follows from Theorem \ref{th:infSDE->local}. \hfill \qedsymbol

\subsection{Existence and uniqueness, $m>2$} \label{se:wellposedness:m>2}
This section proves Theorem \ref{th:intro:wellposed-m>2}. 
Without loss of generality, we may assume $U(0)=\K(0)=0$.
For shorthand, define $a=\inf_x U''(x)$, $b=\inf_x \K''(x)$, and $c=\|\K''\|_\infty$. Because $U$ and $\K$ are even, we have $U'(0)=\K'(0)=0$, and thus
\begin{align}
U(x)\geq \frac{ax^2}{2},\qquad \K(x)\geq \frac{bx^2}{2}, \qquad |\K'(x)| \le c|x|. \label{UKbounds}
\end{align}
In light of the assumption \eqref{asmp:m>2-thm},
\begin{align}
a> m(c-b). \label{a>m(c-b)}
\end{align}

{ \ }

\noindent\textbf{Existence for the fixed point.}
Denote  $d= b-\frac{c}{m-1}$, {and note that $d \le c(m-2)/(m-1) < c$ because $b \le c$}. Consider the  space $\mathcal{S}$ of functions $F \in C^1(\R)$ with absolutely continuous  first derivative and with $d \le F'' \le c$ and satisfying $F(0)=F'(0)=0$. Note that any $F \in \mathcal{S}$ satisfies $dx^2/2 \le F(x) \le cx^2/2$.

Consider the map $T$ defined by
\begin{equation*}
	T(F)(x)=\log\int_\R e^{-U(y)-K(y)-(m-1)F(y)}\d y-\log\int_\R e^{-U(y)-K(x-y)-(m-1)F(y)}\d y.
\end{equation*}
Note that both integrals are finite for every $x \in \R$ if $F\in\mathcal{S}$, because \eqref{a>m(c-b)} implies
\begin{align}
a+b+(m-1)d = a+mb-c>(m-1)c \ge 0, \label{a+b+(m-1)d}
\end{align}
which in combination with \eqref{UKbounds} implies
\begin{align*}
\int_\R e^{-U(y)-K(x-y)-(m-1)F(y)}\d y&\leq\int_\R e^{-\frac{ay^2}{2}-\frac{b(x-y)^2}{2}-\frac{(m-1)dy^2}{2}}\d y<\infty.
\end{align*}
Moreover, any $F \in \mathcal{S}$ automatically satisfies the integrability condition \eqref{fixedpoint-integrability}; indeed,  \eqref{a>m(c-b)} implies $a+md = a+mb-\frac{m}{m-1}c>\frac{m(m-2)}{m-1}c\geq 0$, which yields
\begin{equation*}
\int_\R e^{-U(y)-mF(y)}\d y\leq \int_\R e^{-\frac{ay^2}{2}-\frac{mdy^2}{2}}<\infty.
\end{equation*}
It is then apparent that a fixed point of $T$ in $\mathcal{S}$ is also solution of the fixed point problem in the sense of Definition \ref{def:fixedpoint}.

We then focus on finding a fixed point of $T$ in $\mathcal{S} \subset C(\R)$, which we accomplish by applying Schauder's fixed point theorem.
Equip $C(\R)$ with the topology of uniform convergence on compacts.
It is straightforward to check that $T$ is continuous in $C(\R)$. 
{To see that $\mathcal{S}$ is compact, we express it as $\mathcal{S}=I(\mathcal{S}')$, where $I : C(\R) \to C(\R)$ is defined by $I(G)(x):=\int_0^xG(y)\d y$, and $\mathcal{S}'$ is the set of  $G \in C(\R)$ such that $G(0)=0$ and such that $G(x)-\frac{c+d}{2}x$ Lipschitz with constant $(c-d)/2$. The set $\mathcal{S}'$ is compact by the Arzel\`a-Ascoli theorem, and $I$ is clearly continuous, so $\mathcal{S}$ is compact.}

It remains to show that $T(\mathcal{S})\subset \mathcal{S}$. Let $F_0\in \mathcal{S}$ and define $F=T(F_0)$. That is,
\begin{equation*}
	e^{-F(x)} = \frac{\int_\R e^{-U(y)-K(x-y)-(m-1)F_0(y)}\d y}{\int_\R e^{-U(y)-K(y)-(m-1)F_0(y)}\d y}.
\end{equation*}
For $x \in \R$, let $\rho_0(\cdot|x)$ be the probability density function defined by
\begin{equation*}
	\rho_0(y|x)=\frac{e^{-U(y)-K(x-y)-(m-1)F_0(y)}}{\int_\R e^{-U(y')-K(x-y')-(m-1)F_0(y')}\d y'},
\end{equation*}
Since $F_0 \in \mathcal{S}$, we have $F_0(x) \ge dx^2/2$, which along with \eqref{UKbounds} yields
\begin{align*}
h_x(y) := |K'(x-y)|e^{-U(y)-K(x-y)-(m-1)F_0(y)}&\leq c|x-y| e^{-\frac{ay^2}{2}-\frac{b(x-y)^2}{2}-\frac{(m-1)dy^2}{2}}.
\end{align*}
By \eqref{a+b+(m-1)d}, it holds for each bounded interval $I \subset \R$ that $y \mapsto \sup_{x \in I} h_x(y)$ is integrable. Therefore, we may interchange the derivative and integral to get
\begin{equation*}
\frac{\d}{\d x}\int_\R e^{-U(y)-K(x-y)-(m-1)F_0(y)}\d y=\int_\R -K'(x-y)e^{-U(y)-K(x-y)-(m-1)F_0(y)}\d y.
\end{equation*}
If we define $f(x)$ by
\begin{align*}
	f(x)=\E_{\rho_0(\cdot|x)}\big[K'(x-Y)\big],
\end{align*}
where we understand in the expectation that $Y \sim \rho_0(\cdot|x)$, then
we have
\begin{equation*}
f(x)=\frac{\int_\R K'(x-y)e^{-U(y)-K(x-y)-(m-1)F_0(y)}\d y}{\int_\R e^{-U(y)-K(x-y)-(m-1)F_0(y)}\d y}= F'(x).
\end{equation*}
Note also that $F$ is even because $\K$ is, and thus $f$ is odd, and in particular $f(0)=0$. 
As for the second derivative, we may similarly take derivatives inside the integrals to get
\begin{align*}
f'(x) = & \frac{\int_\R \big( \K''(x-y) - \K'(x-y)^2 \big) e^{-U(y)-K(x-y)-(m-1)F(y)}\d y}{\int_\R e^{-U(y)-K(x-y)-(m-1)F(y)}\d y} \\
	&+ \left(\frac{\int_\R K'(x-y)e^{-U(y)-K(x-y)-(m-1)F(y)}\d y}{ \int_\R e^{-U(y)-K(x-y)-(m-1)F(y)}\d y}\right)^2.
\end{align*}
This can be rewritten more suggestively as
\begin{align*}
f'(x)&= \E_{\rho_0(\cdot|x)}[K''(x-Y)]-\text{Var}_{\rho_0(\cdot|x)}[K'(x-Y)].
\end{align*}
Bounding the variance from below by zero, we have trivially
\begin{align}
f'(x) &\le \E_{\rho_0(\cdot|x)}[K''(x-Y)] \le c. \label{pf:m>2-f'}
\end{align}
To bound the variance from above, we use a Poincar\'e inequality. 
Notice that
\begin{equation*}	
\partial_{yy}(-\log \rho_0(y|x)) =  U''(y)+K''(x-y)+(m-1)F_0''(y) \ge a+b+(m-1)d,
\end{equation*}
for $x,y \in \R$.
Recall that $d=b-\frac{c}{m-1}$ and $a\geq mc-mb$, we get
\begin{equation*}
\inf_{y\in\R}\big(\partial_{yy}(-\log \rho_0(y|x))\big)\geq (m-1)c > 0.
\end{equation*}
By the Brascamp-Lieb inequality \cite[Theorem 4.3]{brascamp2002extensions}, this lower bound implies the following Poincar\'e inequality, for all $x \in \R$:
\begin{equation*}
\text{Var}_{\rho_0(\cdot|x)}[K'(x-Y)]\leq \frac{\E_{\rho_0(\cdot|x)}[K''(x-Y)^2]}{(m-1)c}\leq \frac{c}{m-1}.
\end{equation*}
Apply this in \eqref{pf:m>2-f'}, and recall that $b=\inf \K''$, to get
\begin{equation*}
f'(x) \ge b-\frac{c}{m-1}=d.
\end{equation*}
Hence, $d \le F'' \le c$, and we deduce that $F\in \mathcal{S}$. This shows that $T(\mathcal{S})\subset \mathcal{S}$, completing the proof of existence.

{\ }

\noindent\textbf{Uniqueness for the fixed point.} We next prove uniqueness. Suppose $F$ is a solution of the fixed point problem belonging to $\mathcal{S}$, which we know to exist by the previous part. 
Define the symmetric density function 
\begin{equation*}
	\rho(x,y)=Ce^{-U(x)-U(y)-K(x-y)-(m-1)F(x)-(m-1)F(y)},
\end{equation*} 
where $0 < C < \infty$ is a normalization constant. Suppose $\widetilde{F}$ is another fixed point, with the corresponding density $\widetilde{\rho}$ given by
\begin{equation*}
	\widetilde{\rho}(x,y)=\widetilde{C}e^{-U(x)-U(y)-K(x-y)-(m-1)\widetilde{F}(x)-(m-1)\widetilde{F}(y)}.
\end{equation*}
Using \eqref{asmp:m>2-thm}, it is easy to check that \eqref{asmp:tech2} holds. By Theorem \ref{th:fixedpoint->infSDE}, we know that  
\begin{equation}
F'(x)=\E[K'(x-Y)|X=x], \quad \widetilde{F}'(x)=\widetilde\E[K'(x-Y)|X=x], \label{pf:m>2-uniq-F'}
\end{equation}
where $\E$ and $\widetilde\E$ denote expectation under $\rho$ and $\widetilde{\rho}$, respectively. We will argue by bounding the relative Fisher information,
\begin{align}
I(\widetilde{\rho}\,|\,\rho) &:= \widetilde\E\Big[ \Big|\nabla \log\frac{\widetilde{\rho}}{\rho}(X,Y)\Big|^2\Big] \nonumber \\
	&= (m-1)^2\widetilde\E\big[ (F'(X)-\widetilde{F}'(X))^2 + (F'(Y)-\widetilde{F}'(Y))^2 \big] \nonumber \\
	&= 2(m-1)^2\widetilde\E\big[ (F'(X)-\widetilde{F}'(X))^2\big], \label{pf:m>2-uniq-I1}
\end{align}
with the last line following from symmetry. 
Note that \eqref{pf:m>2-uniq-F'} implies
\begin{align*}
\big(F'(x)-\widetilde{F}'(x)\big)^2 &= \Big(\int_\R \K'(x-y)\big(\rho(y|x) -\widetilde{\rho}(y|x)\big)\,\d y\Big)^2,
\end{align*}
where we define the conditional density
\begin{align*}
\rho(y|x) := \frac{\rho(x,y)}{\int_\R \rho(x,y')\,\d y'},
\end{align*}
and similarly for $\widetilde{\rho}$. The function $\K'$ is Lipschitz with constant $\|\K''\|_\infty=c$, and thus
\begin{align}
\big|F'(x)-\widetilde{F}'(x)\big| &\le cW_1\big(\rho(\cdot|x),\widetilde{\rho}(\cdot|x)\big). \label{pf:m>2-uniq-W1}
\end{align}
Here $W_1$ denotes the Kantorovich distance, defined for two probability densities $(\nu,\mu)$ on $\R$ admitting finite first moment by
\begin{align*}
W_1(\nu,\mu) := \sup_h\int_\R h(x)(\nu(x)-\mu(x))\,\d x,
\end{align*}
where the supremum is over all 1-Lipschitz functions.
Recalling that $F'' \ge d=b-\frac{c}{m-1}$ since $F \in \mathcal{S}$, the conditional density $\rho(\cdot | x)$ satisfies
\begin{align*}
\partial_{yy}(-\log \rho(y|x)) &= U''(y)+K''(x-y)+(m-1)F''(y) \\
	&\ge a+b+(m-1)d = a + mb - c.
\end{align*}
Recall from \eqref{a+b+(m-1)d} that $a + mb - c > (m-1)c \ge 0$. This implies that $\rho(\cdot|x)$ satisfies the log-Sobolev inequality  (by a result of Bakry-\'Emery, see \cite{bakryemery} or \cite[Corollary 5.7.2]{bakryGentilLedoux})
\begin{align*}
H\big(\nu\,|\,\rho(\cdot|x)\big) \le \frac{1}{2(a + mb - c)}I\big(\nu\,|\,\rho(\cdot|x)\big),
\end{align*}
as well as (by a result of Otto-Villani \cite{ottovillani}) the transport inequality
\begin{align*}
W_1^2(\nu,\rho(\cdot|x)) \le \frac{2}{a + mb - c}H\big(\nu\,|\,\rho(\cdot|x)\big),
\end{align*}
for all $x \in \R$ and probability densities $\nu$ on $\R$. Applying these two inequalities in \eqref{pf:m>2-uniq-W1} yields
\begin{align*}
\big(F'(x)-\widetilde{F}'(x)\big)^2 &\le \frac{c^2}{(a + mb - c)^2} I\big(\widetilde{\rho}(\cdot|x)\,|\,\rho(\cdot|x)\big).
\end{align*}
Returning to \eqref{pf:m>2-uniq-I1}, we find
\begin{align}
I(\widetilde{\rho}\,|\,\rho) &\le \frac{2(m-1)^2c^2}{(a + mb - c)^2} \int_\R I\big(\widetilde{\rho}(\cdot|x)\,|\,\rho(\cdot|x)\big) \widetilde{\rho}_X(x)\,\d x, \label{pf:m>2-uniq-I2}
\end{align}
where $\widetilde{\rho}_X(x)=\int_\R \widetilde{\rho}(x,y)\,\d y$ is the marginal density. We rewrite the conditional Fisher information as an unconditional one:
\begin{align*}
2\int_\R I\big(\widetilde{\rho}(\cdot|x)\,|\,\rho(\cdot|x)\big) \widetilde{\rho}_X(x)\,\d x &= 2\int_\R \left(\int_\R \Big( \partial_y \log \frac{\widetilde{\rho}(y|x)}{\rho(y|x)}\Big)^2 \widetilde{\rho}(y|x)\,\d y \right) \widetilde{\rho}_X(x)\,\d x \\
	&= 2\int_{\R^2} \Big( \partial_y \log \frac{\widetilde{\rho}(x,y)}{\rho(x,y)}\Big)^2 \widetilde{\rho}(x,y)\,\d y \d x.
\end{align*}
By symmetry, the right-hand side equals
\begin{align*}
\int_{\R^2} \Big( \Big( \partial_x \log \frac{\widetilde{\rho}(x,y)}{\rho(x,y)}\Big)^2 + \Big( \partial_y \log \frac{\widetilde{\rho}(x,y)}{\rho(x,y)}\Big)^2 \Big)\widetilde{\rho}(x,y)\,\d y \d x = I(\widetilde{\rho}\,|\,\rho).
\end{align*}
Hence, \eqref{pf:m>2-uniq-I2} simplifies to
\begin{align*}
I(\widetilde{\rho}\,|\,\rho) &\le rI(\widetilde{\rho}\,|\,\rho), \quad \text{where} \quad r := \frac{(m-1)^2c^2}{(a + mb - c)^2}.
\end{align*}
Finally, the inequality $a + mb - c > (m-1)c$ from \eqref{a+b+(m-1)d} implies that $r < 1$, which yields $I(\widetilde{\rho}\,|\,\rho)=0$. Hence, $\widetilde{\rho}=\rho$, and by \eqref{pf:m>2-uniq-F'} we have also $F' = \widetilde{F}'$. Thus $F$ and $\widetilde{F}$ are equivalent up to an additive shift.

{\ }

\noindent\textbf{The infinite SDE and local equation.}
We finally explain how to deduce the claims about the infinite SDE and local equation. The claimed uniqueness for the stationary symmetric solution of the local equation follows from Theorem \ref{th:local->fixedpoint}, because $U'$ and $\K'$ are assumed continuous and thus locally bounded, and the claimed uniqueness for the infinite SDE follows from Theorem \ref{th:infSDE->local}. 
Existence under the additional assumption on $U$ will follow from Theorems \ref{th:fixedpoint->infSDE} and \ref{th:infSDE->local}, after we check that \eqref{asmp:tech1} holds.
The inequality $a> m(c-b)$, coming from the assumption \eqref{asmp:m>2-thm}, implies 
\begin{align*}
-U(x)-U(y)-mK(x-y)  &\le -\frac{a}{2}(x^2+y^2)-\frac{mb}{2}(x-y)^2 	\\
	&\le -\frac{mc}{2}\big(x^2+y^2-\frac{2bxy}{c}\big).
\end{align*}
Since also $|b| < c$, we deduce that $\exp(-U(x)-U(y)-m\K(x-y))$ is bounded from above by a constant times a non-degenerate Gaussian. Hence, any power of $|U'(x)| \le c_1 e^{c_2|x|^p}$ is integrable with respect to this density because $0 \le p < 2$. Similarly, any power of $|\K'(x-y)| \le c|x-y|$ is integrable. 
 \hfill \qedsymbol

\section{Linear coefficients and the spectrum of $\T_m$} \label{se:linear-coeff}

In this section, we study the explicitly solvable case of linear coefficients in the infinite SDE. We first  motivate the development by showing how to use the solution to compute the resolvent of $\T_m$, thereby recovering the famous Kesten-McKay law.

\subsection{The resolvent of $\T_m$} \label{se:resolvent}

To a countable locally finite graph one can associate natural infinite-dimensional operators, such as the adjacency operator; see \cite{MoharWoess} for background on the formalism we adopt in this section.
Consider the Hilbert space $\ell_2(\T_m)$ of functions $\T_m \to \R$, with  inner product and norm denoted $\langle \cdot,\cdot\rangle$ and $\|\cdot\|$, respectively. Let $(\bm e_v)_{v\in \T_m}$ denote the natural orthonormal basis, given by $\bm e_v(u) = \delta_{uv}$. 

Let $\bm I : \R^{\T_m} \to \R^{\T_m}$ denote the identity operator. Let $\bm{A}:\R^{\T_m}\to \R^{\T_m}$ denote the adjacency operator, defined by
\begin{equation*}
	(\bm{A} \bm{x})(v) =\sum_{u\in N(v)}x_u, \qquad \text{for } \bm{x}=(x_v)_{v \in \T_m} \in \R^{\T_m}.
\end{equation*}
Note that $\bm{A}$ maps $\ell_2(\T_m)$ into itself and defines a bounded operator thereon, because Cauchy-Schwarz implies
\begin{align*}
\|\bm{A} \bm{x}\|^2 &= \sum_{v \in \T_m} \Big(\sum_{u\in N(v)}x_u\Big)^2 \le \sum_{v \in \T_m} m\sum_{u\in N(v)}x_u^2 = m^2 \sum_{u \in \T_m} x_u^2.
\end{align*}
Moreover, $\bm{A}$ is symmetric, because $\langle \bm{e}_v, \bm{A} \bm{e}_u\rangle$ equals $1$ if $u$ and $v$ are neighbors and zero otherwise.  The spectrum of $\bm{A} : \ell_2(\T_m)\to \ell_2(\T_m)$ is well known (due to \cite{kesten1959symmetric}) to be
\begin{align}
\sigma(\bm{A}) = [-2\sqrt{m-1},2\sqrt{m-1}]. \label{sigma(A)}
\end{align}

Our focus in this section is the \emph{resolvent}, $z \mapsto \langle \bm e_v, (z\bm I-\bm A)^{-1}\bm e_v \rangle$, of the operator $\bm{A}$ on $\ell_2(\T_m)$, well-defined for $v \in \T_m$ and $z\in \R \setminus \sigma(\bm A)$. {It can be argued that this quantity does not depend on the choice of vertex $v \in \T_m$, because the graph $\T_m$ is vertex transitive.} By the spectral theorem, there exists a Borel probability measure $\mu_{\T_m}$ on $\R$ such that
\begin{align*}
\int_{\R} \frac{1}{z-r}\,\mu_{\T_m}(\d r) = \langle \bm e_v, (z\bm I-\bm A)^{-1}\bm e_v \rangle, \qquad z \in \R \setminus \sigma(\bm{A}).
\end{align*}
We call $\mu_{\T_m}$ the \emph{spectral measure} of the tree $\T_m$.
A famous result originating from the work of Kesten \cite{kesten1959symmetric} and McKay \cite{mckay1981expected} identifies this spectral measure as
\begin{align}
\mu_{\T_m}(\d x) &= \frac{m\sqrt{(4(m-1)-x^2)_+}}{2\pi(m^2-x^2)}\,\d x. \label{KestenMcKay}
\end{align}
Later, a proof was given by the resolvent method in \cite{bordenave2010resolvent}, where they identified 
\begin{align}
\langle \bm e_v, (z\bm I-\bm A)^{-1}\bm e_v \rangle &= \frac{2(m-1)}{(m-2)z +m\sqrt{z^2-4(m-1)}}, \label{KestenMcKay-resolvent}
\end{align}
from which one can recover the formula \eqref{KestenMcKay} via inverse Stieltjes transform.

We will show how the formula \eqref{KestenMcKay-resolvent} arises from certain versions of the infinite SDE \eqref{intro:SDE} with linear coefficients $U'$ and $\K'$. To properly derive \eqref{KestenMcKay} from \eqref{KestenMcKay-resolvent} via inverse Stieltjes transform would require knowledge of \eqref{KestenMcKay-resolvent} for complex $z$ in the upper half-plane.
For simplicity, we will limit our attention to recovering \eqref{KestenMcKay-resolvent} for real numbers $z > m$, but the extension to complex $z$ is not difficult.

\subsection{Connection with the infinite SDE} \label{se:infSDE-linear}

Consider $U(x)=(z-m)x^2/2$ and $\K(x)=x^2/2$, where $z \in \R$ is fixed. The infinite SDE \eqref{intro:SDE} becomes
\begin{align}
\d X^v_t &= -\Big((z-m)X^v_t + \sum_{u \in N(v)}(X^v_t-X^u_t)\Big) \d t + \sqrt{2}\,\d W^v_t \nonumber \\
	&= -\Big(z X^v_t - \sum_{u \in N(v)}X^u_t\Big) \d t + \sqrt{2}\,\d W^v_t, \qquad v \in \T_m. \label{SDE-linear}
\end{align}
The well-posedness of this SDE, an infinite-dimensional Ornstein-Uhlenbeck equation, is somewhat subtle, depending on what space one wants the process $\bm{X}_t$ to reside in.
In operator form, we may express it as
\begin{align*}
\d \bm{X}_t &= -(z\bm{I} - \bm{A})\bm{X}_t\, \d t + \sqrt{2} d\bm{W}_t,
\end{align*}
where $\bm{W}=(W^v)_{v \in \T_m}$ are independent Brownian motions. Note this is an equality between processes with values in $\R^{\T_m}$, not $\ell_2(\T_m)$, because $\bm{W}_t$ almost surely fails to reside in $\ell_2(\T_m)$.
Ignoring these subtleties, one formally expects the invariant measure of this Ornstein-Uhlenbeck process to be the centered Gaussian with covariance operator given by $(z\bm{I}-\bm{A})^{-1}$. This is heuristically why the resolvent should appear, but making this precise requires a more careful discussion of integrability issues.

We make a definition for a general random variable $\bm{Y} =(Y_v)_{v \in \T_m}$ in $\R^{\T_m}$.
Let $\R^{\T_m}_{\mathrm{fin}}$ denote the space of vectors with at most finitely many nonzero coordinates.
For any $\bm{x} \in \R^{\T_m}_{\mathrm{fin}}$, the inner product $\langle \bm{x},\bm{Y}\rangle := \sum_{v \in \T_m} x_v Y_v$ is well-defined.
Let us say that $\bm{Y}$ is \emph{$\ell_2(\T_m)$-extendable} if there exists a constant $c \ge 0$ such that
\begin{align}
\E[\langle \bm{x} ,\bm{Y}\rangle^2]  \le c\|\bm{x}\|^2, \qquad \forall \bm{x}=(x_v)_{v \in \T_m} \in \R^{\T_m}_{\mathrm{fin}}. \label{condition}
\end{align}
Recall that $\T_m^k$ denotes the ball of radius $k$ around an arbitrarily fixed root vertex in $\T_m$.
If $\bm{Y}$  is $\ell_2(\T_m)$-extendable, then for any $\bm{x} \in \ell_2(\T_m)$ 
\begin{align*}
\langle \bm{x}|_{\T_m^k},\bm{Y}\rangle = \sum_{v\in \T_m^k} x_v Y_v  
\end{align*}
is easily seen to define a Cauchy sequence in $L^2$.
We may thus define a random variable $\langle \bm{x},\bm{Y}\rangle$ as the $L^2$ limit of $\langle \bm{x}|_{\T_m^k},\bm{Y}\rangle$ as $k\to\infty$, and note that $\E[\langle \bm{x},\bm{Y}\rangle^2] \le c\|\bm{x}\|_{\ell_2(\T_m)}^2$.
\begin{proposition}\label{pr:infSDE-resolvent}
Let $z \in \R \setminus \sigma(\bm{A})$.
Let $\bm X=\big(X^v\big)_{v\in \T_m}$ be a SHM solution of the infinite SDE \eqref{SDE-linear}.
If $\bm{X}_0$ is $\ell_2(\T_m)$-extendable, then 
\begin{equation*}
\langle \bm{e}_v, (z \bm I-\bm A)^{-1} \bm{e}_v\rangle = \E\big[(X_0^v)^2\big], \qquad \forall v \in \T_m.
\end{equation*}
\end{proposition}
\begin{proof}
Since $\bm{X}_0$ is $\ell_2(\T_m)$-extendable, we may define the square-integrable random variable $\langle \bm x,\bm X_0\rangle$ for each $\bm{x} \in \ell_2(\T_m)$ as above.
Define the covariance operator $p(\cdot,\cdot)$ by
\begin{equation*}
p(\bm x,\bm y) = \E\big[\langle \bm x,\bm X_0\rangle\langle \bm y,\bm X_0\rangle\big].
\end{equation*}
We first claim that operator $\bm A$ is symmetric with respect to $p$, or $p(\bm A\bm x,\bm y) = p(\bm x,\bm A\bm y)$.
To do so, it suffices to show that for any distinct $v,u\in \T_m$, 
\begin{equation}\label{adjoint2}
p(\bm A\bm e_v,\bm e_u)=p(\bm e_v,\bm A\bm e_u).
\end{equation}
Noting that $\langle \bm{e}_u,\bm{X}_0\rangle = X^u_0$ and $\langle \bm A\bm e_v,\bm{X}_0\rangle = \sum_{w \in N(v)} X^w_0$, 
we find that the left-hand side of \eqref{adjoint2} equals
\begin{align*}
\sum_{w \in N(v)}g(d(u,w)),
\end{align*}
where $d$ is the graph distance and $g(k)$ is defined to be $\E[X^{r_1}_0 X^{r_2}_0]$ for any two vertices $r_1,r_2$ with distance $k$; this is well-defined by homogeneity. If $d(v,u)=k>0$, then exactly one neighbor of $v$ is at distance $k-1$ from $u$, and the remaining $m-1$ neighbors of $v$ are at distance $k+1$ from $u$. Hence, the left-hand side of \eqref{adjoint2} becomes $g(k-1)+(m-1)g(k+1)$. This is clearly symmetric in $(v,u)$, and \eqref{adjoint2} follows. 

For any  $\bm{x} \in \R^{\T_m}_{\mathrm{fin}}$, the SDE \eqref{SDE-linear} implies
\begin{align*}
\d \langle \bm{x},\bm{X}_t\rangle = -\langle (z\bm{I}-\bm{A})\bm{x},\bm{X}_t\rangle \,\d t + \sqrt{2} \, \d \langle \bm{x}, \bm{W}_t\rangle ,
\end{align*}
The covariation process of  $\langle \bm{x}, \bm{W}_t\rangle$ and $\langle \bm{y}, \bm{W}_t\rangle$ is $t \langle \bm{x},\bm{y}\rangle$, and thus It\^o's product rule yields
\begin{align*}
\d \langle \bm{x},\bm{X}_t\rangle \langle \bm{y},\bm{X}_t\rangle &= \Big(-\langle \bm{x},\bm{X}_t\rangle \langle (z\bm{I}-\bm{A})\bm{y},\bm{X}_t\rangle -\langle \bm{y},\bm{X}_t\rangle \langle (z\bm{I}-\bm{A})\bm{x},\bm{X}_t\rangle + 2 \langle \bm{x},\bm{y}\rangle \Big)\,\d t \\
	&\qquad + \sqrt{2} \langle \bm{x},\bm{X}_t\rangle \, \d \langle \bm{y}, \bm{W}_t\rangle + \sqrt{2} \langle \bm{y},\bm{X}_t\rangle \, \d \langle \bm{x}, \bm{W}_t\rangle.
\end{align*}
Taking expectations and using stationarity, we find
\begin{align*}
0 &= \E\Big[-\langle \bm{x},\bm{X}_t\rangle \langle (z\bm{I}-\bm{A})\bm{y},\bm{X}_t\rangle -\langle \bm{y},\bm{X}_t\rangle \langle (z\bm{I}-\bm{A})\bm{x},\bm{X}_t\rangle \Big]  + 2 \langle \bm{x},\bm{y}\rangle \\
	&= -2p(\bm{x},(z\bm{I}-\bm{A})\bm{y}) + 2 \langle \bm{x},\bm{y}\rangle ,
\end{align*}
where the last step used the symmetry of $\bm{A}$ with respect to $p(\cdot,\cdot)$.
Note that $p(\cdot,\cdot)$ is well-defined and continuous on $\ell_2(\T_m)^2$ thanks to the assumed $\ell_2(\T_m)$-extendability.
The previous identity thus extends from $\R^{\T_m}_{\mathrm{fin}}$ to $\ell_2(\T_m)$.
Choose $\bm x=\bm e_v$ and $\bm y=(z\bm I-\bm A)^{-1}\bm e_v$ to find
\begin{equation*}
{ \langle \bm e_v, (z\bm I-\bm A)^{-1}\bm e_v\rangle = p(\bm{e}_v,\bm{e}_v) = \E\big[\langle \bm e_v,\bm X_0\rangle ^2\big]=\E\big[(X_0^v)^2\big], }
\end{equation*}
for any $v \in \T_m$, which completes the proof.
\end{proof}

\begin{remark}
The identity $\langle \bm{x},\bm{y}\rangle = p(\bm{x},(z\bm{I}-\bm{A})\bm{y})$ shown in the proof of Proposition \ref{pr:infSDE-resolvent} is true for all finitely supported vectors $\bm{x},\bm{y} \in \R^{\T_m}$, even if the SHM solution $\bm{X}_t$ is not $\ell_2(\T_m)$-extendable. But the vector $\bm y=(z\bm I-\bm A)^{-1}\bm e_v$ used in the final step is not finitely supported and thus can not be safely plugged into this identity without $\ell_2(\T_m)$-extendability.
\end{remark}

\subsection{Some generalities on $\T_m$-indexed Gaussians}

We deal in this section with some generalities involving Gaussian measures on $\R^{\T_m}$.
A random variable $\bm{Y}=(Y_v)_{v \in \T_m}$ in $\R^{\T_m}$ is Gaussian if its finite-dimensional distributions are Gaussian. We say $\bm{Y}$ is a \emph{Markov random field} if $\bm{Y}_A$ and $\bm{Y}_B$ are conditionally independent given $\bm{Y}_S$, for any finite disjoint sets $A,B,S \subset \T_m$ with the property that every path from $A$ to $S$ passes through $S$. We say $\bm{Y}$ is homogeneous if $(Y_v)_{v \in \T_m} \stackrel{d}{=} (Y_{\varphi(v)})_{v \in \T_m}$ for every automorphism $\varphi$ of $\T_m$.
Centered Gaussians with both of these properties have a very simple structure:

\begin{lemma} \label{le:Gaussians}
Let $\bm{Y}$ be a centered Gaussian in $\R^{\T_m}$ which is a homogeneous Markov  random field. Then there exist $\sigma^2 \ge 0$ and $\rho \in [-1,1]$ such that
\begin{align*}
\E[Y_v^2]=\sigma^2, \qquad \E[Y_vY_u] = \sigma^2 \rho^{d(u,v)}, \qquad \forall u,v \in \T_m,
\end{align*}
where $d(u,v)$ denotes the graph distance; we call $\sigma^2$ the \emph{variance} and $\rho$ the \emph{correlation} of $\bm{Y}$.
Finally, if $\sigma^2 > 0$, then $\bm{Y}$ is $\ell_2(\T_m)$-extendable if and only if $|\rho| < 1/\sqrt{m-1}$.
\end{lemma}

We defer the proof to Section \ref{se:Gaussian-proofs}, preferring first to complete the discussion of the infinite SDE, which we analyze via the corresponding local equation.

\subsection{Solution via the local equation} 

In the following, define
\begin{align*}
\rho_{\pm} &= \frac{z \mp \sqrt{z^2-4(m-1)}}{2(m-1)}, \qquad \sigma^2_{\pm} = \frac{m\sqrt{z^2-4(m-1)}\mp(m-2)z}{\pm 2(z^2-m^2)},
\end{align*}
which are real and well-defined when $z \ge 2\sqrt{m-1}$ and $z \neq m$. {The pairs $(\sigma^2_+,\rho_+)$ and $(\sigma^2_-,\rho_-)$ are the solutions of the pair of equations
\begin{equation}\label{equ:rho}
 (m-1)\rho^2-z\rho +1=0, \qquad \sigma^2(z-m\rho)=1. 
\end{equation}
Notably, these are exactly the two equations found in \cite[Theorem 2 and Example 1]{bordenave2010resolvent} to characterize the resolvent of the $m$-regular tree $\T_m$, using the notation $X(-z)=\sigma^2$ and $Y(-z)=\rho$.}

\begin{theorem} \label{th:linear}
There are five cases:
\begin{enumerate}[(i)]
\item Let $z > m$. Then there exists a SHM solution $\bm{X}$ of the infinite SDE \eqref{SDE-linear}, and it is unique in law among those solutions for which the function $x \mapsto \E[X^u_t\,|\,X^v_0=x]$ belongs to $L^q_{\mathrm{loc}}(\R)$ for some $q > 2$ and for some edge $(v,u)$. Moreover, $\bm{X}_0$ is a centered Gaussian which is a homogeneous Markov random field, and it has variance $\sigma^2_+$ and correlation $\rho_+$.
\item Let $z=m>2$. Then there exist a unique Gaussian SHM solution $\bm{X}$ of the infinite SDE \eqref{SDE-linear}. It holds that $\bm{X}_0$ is a centered Gaussian which is an homogeneous Markov random field with variance $\sigma^2 = \frac{m-1}{m(m-2)}$ and correlation $\rho=1/(m-1)$. Moreover, $\bm{X}_0$ is $\ell_2(\T_m)$-extendable.
\item Let $2\sqrt{m-1} < z < m$.  Then there exist exactly two SHM solutions $\bm{X}^{\pm}$ of the infinite SDE \eqref{SDE-linear}. It holds that $\bm{X}_0^{\pm}$ is a centered Gaussian which is a  homogeneous Markov random field with variance $\sigma^2_{\pm}$ and correlation $\rho_{\pm}$. Moreover, $\bm{X}^+_0$ is $\ell_2(\T_m)$-extendable, whereas $\bm{X}^-_0$ is not.
\item Let $z=2\sqrt{m-1}$, $m\neq 2$.  Then there exist a unique Gaussian SHM solution $\bm{X}$ of the infinite SDE \eqref{SDE-linear}. It holds that $\bm{X}_0$ is a centered Gaussian which is a homogeneous Markov random field with variance $\sigma^2 = 1/((m-1)^{1/2} - (m-1)^{-1/2})$ and correlation $\rho=1/\sqrt{m-1}$. Moreover, $\bm{X}_0$ is not $\ell_2(\T_m)$-extendable.
\item Let $z < 2\sqrt{m-1}$ or $z=m=2$. Then there is no Gaussian SHM solution of the infinite SDE \eqref{SDE-linear}.
\end{enumerate}
\end{theorem}
\begin{proof}
With $U(x)=(z-m)x^2/2$ and $\K(x)=x^2/2$,  the local equation \eqref{intro:localeq} becomes
\begin{align}
\begin{split}
\d X_t &= (-zX_t+Y_t+(m-1)\E[Y_t|X_t]\big) \,\d t + \sqrt{2} \d W_t, 	\\
\d Y_t &= (X_t-zY_t+(m-1)\E[X_t|Y_t]\big) \,\d t + \sqrt{2} \d B_t.  
\end{split} \label{local-linear}
\end{align}
When $z > m$, all of the assumptions of Theorem \ref{th:intro:wellposed-m>2} hold, and this implies the existence and uniqueness claimed in (i). Otherwise, there is still a one-to-one correspondence between Gaussian solutions of the infinite SDE \eqref{SDE-linear} and Gaussian solutions of the local equation \eqref{local-linear}, as follows.
A Gaussian SHM solution $\bm{X}$ gives rise to a solution of the local equation, in the sense of Theorem \ref{th:infSDE->local}.
Conversely, a Gaussian solution $(X,Y)$ of the local equation  gives rise to a solution of the fixed point problem, as follows. Note that $f(x) = \E[\K(X_0-Y_0)\,|\,X_0=x]=\rho x$, where $\rho$ is the correlation of $X_0$ and $Y_0$. Clearly, $U'$, $\K'$, and $f$ belong to $L^q_{\mathrm{loc}}(\R)$ for any $q > 2$, so Theorem \ref{th:local->fixedpoint} applies to yield that $F(x)=\int_0^x f(y)\,\d y = \rho x^2/2$ solves the fixed point problem, and the density of $(X_0,Y_0)$ takes the form $\rho(x,y)$ specified in \eqref{jointdensity-thm}.  The integral in \eqref{asmp:tech1'} rewrites as
\begin{align*}
\int_{\R^2} (|U'(x)|^p + |\K'(x-y)|^p)\rho(x,y)\,\d x \d y = \E\big[ |(z-m)X_0|^p + |X_0-Y_0|^p \big],
\end{align*}
which is finite because $(X_0,Y_0)$ is Gaussian. We may then apply Theorem \ref{th:fixedpoint->infSDE} to find a SHM solution $\bm{X}=(X^v)_{v \in \T_m}$ of the infinite SDE \eqref{SDE-linear} such that $(X^v_0,X^u_0) \stackrel{d}{=} (X_0,Y_0)$ for each edge $(v,u)$.

Now, suppose $(X,Y)$ is a stationary symmetric  solution of the local equation such that $(X_0,Y_0)$ is Gaussian.
Define the variance $\sigma^2:=\Var(X_0)$ and correlation $\rho:=\Cov(X_0,Y_0)/\Var(X_0)$. Then $\E[X_t\,|\,Y_t]=\rho Y_t$ and $\E[Y_t\,|\,X_t]=\rho X_t$, and so $(X,Y)$ satisfies
\begin{equation*}
	\d\left(\begin{matrix} X_t\\ Y_t\end{matrix}\right)=-\left(\begin{matrix}z-(m-1)\rho & -1 \\ -1 & z-(m-1)\rho\end{matrix}\right)\left(\begin{matrix} X_t\\ Y_t\end{matrix}\right)\d t+\sqrt{2}\,\d \left(\begin{matrix} W_t\\ B_t\end{matrix}\right).
\end{equation*}
This is a multivariate Ornstein-Uhlenbeck process, and its unique invariant measure is the centered Gaussian with covariance matrix given by the inverse of
\begin{align}
\left(\begin{matrix}z-(m-1)\rho & -1 \\ -1 & z-(m-1)\rho\end{matrix}\right), \label{OUcov}
\end{align}
assuming this matrix is positive definite. But this covariance matrix must match that of $(X_0,Y_0)$, which means we must have
\begin{equation}
\left(\begin{matrix}z-(m-1)\rho & -1 \\ -1 & z-(m-1)\rho\end{matrix}\right) 	\sigma^2\left(\begin{matrix}1 & \rho \\ \rho & 1\end{matrix}\right)=\left(\begin{matrix}1 & 0 \\ 0 & 1\end{matrix}\right). \label{rho-matrix}
\end{equation}
Note if $-1 < \rho < 1$ that the matrix
\begin{align*}
\sigma^2\left(\begin{matrix}1 & \rho \\ \rho & 1\end{matrix}\right)
\end{align*}
is automatically positive definite, and so the positive definiteness of \eqref{OUcov} follows from the identity \eqref{rho-matrix}.
The equation \eqref{rho-matrix} is equivalent to the two equations \eqref{equ:rho}. The solutions of \eqref{equ:rho} are given by $(\sigma^2_{\pm},\rho_{\pm})$, but they are only meaningful if they belong to $S := (0,\infty) \times (-1,1)$. Let us discuss each case separately:
\begin{enumerate}[(i)]
\item If $z > m$, then $(\sigma_+^2,\rho_+)$ is the unique solution of \eqref{equ:rho} belonging to $S$. The other solution $\rho_-$ of the first equation  of \eqref{equ:rho} is larger than $1$.
\item If $z=m$, then the two solutions of the first equation of \eqref{equ:rho} are $\rho_+=1/(m-1)$ and $\rho_-=1$. For $\rho=1$ and $z=m$, there is no solution of the second equation of \eqref{equ:rho}. For $\rho=1/(m-1)$, the second equation of \eqref{equ:rho} becomes
\begin{align*}
1=\sigma^2  m(1-\rho) = \sigma^2 \frac{m(m-2)}{m-1}.
\end{align*}
For $m=2$ there is again no solution.
\item If $2\sqrt{m-1} < z < m$, then both $(\sigma^2_+,\rho_+)$ and $(\sigma^2_-,\rho_-)$ are solutions belonging to $S$, and we always have $\rho_+ < 1/\sqrt{m-1} < \rho_-$. Indeed, $\rho_+$ is easily seen to be decreasing in $z$, whereas $\rho_-$ is increasing in $z$, and they both agree and equal $1/\sqrt{m-1}$ when $z=2\sqrt{m-1}$. In light of Lemma \ref{le:Gaussians}, this proves the claimed $\ell_2(\T_m)$-extendability.
\item If $z=2\sqrt{m-1}$ and $m > 2$, then $\rho_+=\rho_-=1/\sqrt{m-1}$ and $\sigma_+^2=\sigma_-^2=1/((m-1)^{1/2} - (m-1)^{-1/2})$ comprise the unique solution in $S$.
\item If $z < 2\sqrt{m-1}$, then there are no real solutions of \eqref{equ:rho}.
\end{enumerate}
\end{proof}

\subsection{The resolvent, revisited}

Equipped with Theorem \ref{th:linear}, we can now compute the resolvent. For $z > m$, we know from Theorem \ref{th:linear}(i) that there is a Gaussian SHM solution $\bm{X}$ of the infinite SDE \eqref{SDE-linear}, and it satisfies
\[
\E[(X^v_0)^2]=\sigma_+^2=\frac{m\sqrt{z^2-4(m-1)}-(m-2)z}{ 2(z^2-m^2)}.
\]
Moreover, $\bm{X}_0$ is $\ell_2(\T_m)$-extendable, and thus Proposition \ref{pr:infSDE-resolvent} implies that $\E[(X^v_0)^2]=\langle \bm{e}_v, (z \bm I-\bm A)^{-1} \bm{e}_v\rangle$.
Equating these two identities recovers the resolvent formula \eqref{KestenMcKay-resolvent}.

\subsection{Proof of Lemma \ref{le:Gaussians}} \label{se:Gaussian-proofs}
The fact that $\E[Y_v^2]=:\sigma^2$ does not depend on $v$ follows from homogeneity, as does the fact that the correlation $\rho=\E[Y_vY_u]/\sigma^2$ is the same for every edge $(v,u)$. Suppose, for induction, that $\E[Y_vY_u] = \sigma^2 \rho^{d(u,v)}$ holds for all $u,v$ at distance $k$ or less, for some integer $k \ge 0$. Let $v,u$ have distance $k+1$, and let $w$ be the unique vertex at distance $k$ from $u$ and $1$ from $v$. Then $\E[Y_wY_u] = \sigma^2 \rho^k$, and $Y_v$ and $Y_w$ are jointly Gaussian with equal variance and with correlation $\rho$. In particular, we have $\E[Y_v\,|\,Y_w]=\rho Y_w$. By the Markov property,
\begin{align*}
\E[Y_vY_u] &= \E\big[\E[Y_vY_u\,|\,Y_w]\big] = \E\big[\E[Y_v\,|\,Y_w]Y_u\big] = \rho\E[Y_wY_u]  = \sigma^2\rho ^{k+1}.
\end{align*}

It remains to prove the claim about $\ell_2(\T_m)$-extendability.
In the following, fix an arbitrary root vertex $\o \in \T_m$. Let $|v|:=d(v,\o)$ denote the \emph{generation} of $v$. Let $\T_m^k=\{v \in \T_m : |v| \le k\}$ and $\partial \T_m^k = \{v \in \T_m : |v|=k\}$.  An \emph{ancestor} of a vertex $v$ is any vertex lying on the unique shortest path connecting $v$ to $\o$. For two vertices $v$ and $u$, there is a unique ancestor $v \wedge u$ of maximal generation. Note that
\begin{align}
d(u,v) = |u| + |v| - 2|u \wedge v|, \qquad \forall u,v \in \T_m. \label{treemetric}
\end{align}

First assume $|\rho| < 1/\sqrt{m-1}$. For $\bm{x} \in \R^{\T_m}_{\mathrm{fin}}$, 
\begin{align*}
\E[\langle \bm{x},\bm{Y}\rangle^2] &= \sum_{u,v \in \T_m} x_v x_u \E[Y_v Y_u] = \sum_{u,v \in \T_m} x_v x_u \sigma^2 \rho^{d(u,v)}.
\end{align*}
This is bounded from above for any $0 < \epsilon \le 1$ by
\begin{align*}
\frac{\sigma^2}{2}\sum_{u,v \in \T_m} \rho^{d(u,v)}\Big( x_u^2  \epsilon^{|v|-|u|} + x_v^2  \epsilon^{|u|-|v|}\Big) = \sigma^2 \sum_{u,v \in \T_m} \rho^{d(u,v)} x_u^2  \epsilon^{|v|-|u|}.
\end{align*}
For $u \in \T_m$ with $|u|=i$, using \eqref{treemetric} we can write
\begin{align*}
\sum_{v \in \T_m} \rho^{d(u,v)} \epsilon^{|v|-|u|} &= \sum_{j=0}^i \sum_{\ell=j}^\infty \sum_{v : |v \wedge u|=j, \, |v|=\ell} \rho^{\ell+i-2j} \epsilon^{\ell-i} \\
	&= \sum_{j=0}^i \sum_{\ell=j}^\infty \rho^{\ell+i-2j} \epsilon^{\ell-i} (m-1)^{\ell-j},
\end{align*}
with the last identity following from the simple observation that there are exactly $(m-1)^{\ell-j}$ choices of $v \in \T_m$ satisfying $|v \wedge u|=j$ and $|v|=\ell$. Thus 
\begin{align*}
\sum_{v \in \T_m} \rho^{d(u,v)} \epsilon^{|v|-|u|} &=\sum_{j=0}^i \frac{(\rho/\epsilon)^{i-j}}{1-\epsilon \rho(m-1)}.
\end{align*}
As long as $|\rho/\epsilon|<1$ and $|\epsilon\rho(m-1)|<1$, this can be computed as
\begin{align*}
= \sum_{j=0}^i \frac{(\rho/\epsilon)^{i-j}}{1-\epsilon \rho(m-1)} \le (1-\epsilon \rho (m-1))^{-1}(1-(\rho/\epsilon))^{-1} < \infty
\end{align*}
In other words, this is valid if we can choose $|\rho| < \epsilon < 1/|\rho(m-1)|$, which is the case if $\rho^2(m-1)<1$. In this case, putting it all together, we obtain
\begin{align*}
\E[\langle \bm{x},\bm{Y}\rangle^2] \le \sigma^2 \|\bm{x}\|^2(1-\epsilon \rho (m-1))^{-1}(1-(\rho/\epsilon))^{-1}, \qquad \forall \bm{x} \in \R^{\T_m}_{\mathrm{fin}},
\end{align*}
which proves $\ell_2(\T_m)$-extendability. 

To prove the converse, assume now that $|\rho| \ge 1/\sqrt{m-1}$. Define $\bm{x} \in \R^{\T_m}$ by $x_v = a^{|v|}$, for some constant $a$ with $|a| < 1/\sqrt{m-1}$ to be chosen later. Noting that $|\partial \T_m^k|=m(m-1)^{k-1}$ for $k \ge 1$, we find
\begin{align}
\|\bm{x}\|^2 = \sum_{v \in \T_m} a^{2|v|} = 1 + \sum_{k=1}^\infty  m(m-1)^{k-1} a^{2k}  = 1 + \frac{m}{m-1}\frac{1}{1-b^2} < \infty, \label{xnorm1}
\end{align}
where we set $b := |a|\sqrt{m-1}$.
Moreover, recalling \eqref{treemetric}, we have
\begin{align*}
\E[\langle \bm{x}|_{\T_m^k},\bm{Y} \rangle^2] &=  \sigma^2 \sum_{u,v \in \T_m^k} a^{|v|+|u|} \rho^{d(u,v)} \\
	&= \sigma^2 \sum_{u \in \T_m^k} \sum_{j=0}^{|u|} \sum_{\ell=j}^k \sum_{v : |v \wedge u|=j, \, |v|=\ell} a^{|v|+|u|} \rho^{|u|+|v|-2|v \wedge u|} \\
	&= \sigma^2 \sum_{u \in \T_m^k} \sum_{j=0}^{|u|} \sum_{\ell=j}^k a^{\ell+|u|} \rho^{|u|+\ell -2j} (m-1)^{\ell-j},
\end{align*}
where we use, as before, the fact that there are $(m-1)^{\ell-j}$ choices of $v$ satisfying $|v \wedge u|=j$ and $|v|=\ell$. Since also $|\partial \T_m^i|=m(m-1)^{i-1}$ for $i \ge 1$, this becomes
\begin{align*}
 &= \sigma^2 \sum_{\ell=0}^k a^{\ell} \rho^{\ell} (m-1)^{\ell} + \frac{m\sigma^2}{m-1} \sum_{i=1}^k  \sum_{j=0}^{i} \sum_{\ell=j}^k  a^{\ell+i} \rho^{\ell+i-2j}(m-1)^{\ell+i-j}.
\end{align*}
Let us choose $a$ to have the same sign as $\rho$.
{Then the first term is non-negative, and using $|\rho| \ge 1/\sqrt{m-1}$ the second term is bounded from below by
\begin{align*}
\frac{m\sigma^2}{m-1} \sum_{i=1}^k  \sum_{j=0}^{i} \sum_{\ell=j}^k  |a|^{\ell+i}  (m-1)^{(\ell+i)/2} &\ge \frac{m\sigma^2}{m-1} \sum_{i=1}^k  \sum_{\ell=0}^k  |a|^{\ell+i}  (m-1)^{(\ell+i)/2}  \\
	&\ge \frac{m\sigma^2}{m-1} \bigg( \sum_{i=1}^k  |a|^i(m-1)^{i/2}\bigg)^2 \\
	&= \frac{m\sigma^2}{m-1} \bigg(\frac{b- b^{k+1}}{1-b}\bigg)^2 ,
\end{align*}
}
where we recall that $b:=|a|\sqrt{m-1} < 1$. Thus, recalling \eqref{xnorm1}, we have shown that
\begin{align*}
\liminf_{k \to \infty}\frac{\E[\langle \bm{x}|_{\T_m^k},\bm{Y} \rangle^2]}{\|\bm{x}|_{\T_m^k}\|^2}  \ge \frac{\frac{m\sigma^2}{m-1} \frac{ {b} }{(1-b)^2}}{1 + \frac{m}{m-1}\frac{1}{1-b^2}} = {b} \sigma^2 \bigg( \frac{m-1}{m}(1-b)^2 + \frac{1-b}{1+b} \bigg)^{-1}.
\end{align*}
Sending $|a| \uparrow 1/\sqrt{m-1}$, with $a$ having the same sign as $\rho$, yields  $b \uparrow 1$, and the right-hand side diverges. Hence, $\bm{Y}$ is not $\ell_2(\T_m)$-extendable. \hfill\qedsymbol

\section{Particle systems with repulsion} \label{se:repulsion}

In this section we study the case $\K(x) = -\beta\log |x|$, with  resulting infinite SDE system given by \eqref{DysonSDE}.
The corresponding local equation is
\begin{align}
\begin{split}
\d X_t &= \Big(-U'(X_t) + \frac{\beta}{X_t-Y_t} + \beta(m-1)\E\Big[\frac{1}{X_t-Y_t}\,\big|\,X_t\Big]\Big)\d t + \sqrt{2} \, \d W_t, \\
\d Y_t &= \Big(-U'(Y_t) + \frac{\beta}{Y_t-X_t} + \beta(m-1)\E\Big[\frac{1}{Y_t-X_t}\,\big|\,Y_t\Big]\Big)\d t + \sqrt{2} \,\d B_t.
\end{split} \label{Dyson-localeq}
\end{align}
In the case $\beta=2$, we show in this section that the associated fixed point problem has a unique solution, which we explicitly identify. Using Theorems \ref{th:fixedpoint->infSDE} and \ref{th:infSDE->local}, this yields also the existence of a SHM solution of the infinite SDE \eqref{DysonSDE}, with an explicit form for the joint law of $(X^v_0,X^u_0)$ for any edge $(u,v)$,  and a stationary symmetric solution of the local equation \eqref{Dyson-localeq}.
We are unable to claim uniqueness for \eqref{DysonSDE} or \eqref{Dyson-localeq}, because $\K$ violates the local integrability assumption of Theorem \ref{th:local->fixedpoint}.

\begin{theorem} \label{th:Dyson}
Suppose $m \ge 2$. Let $\K(x)=-2 \log |x|$, and suppose $U$ is an even function such that $s_k := \int_\R x^k e^{-U(x)}\,\d x < \infty$ for $k \ge 0$.
Then there is a unique solution $F$ of the fixed point problem up to additive shifts, and it is given by
\begin{align}
F(x) = -\log(x^2+r), \label{Dyson-Fsolution}
\end{align}
where $r$ is the unique positive real root of the polynomial
\begin{align}
\sum_{j=0}^{m}\frac{m - 2j}{m}\binom{m}{j}s_{2j}r^{m-j}. \label{Dyson-polynomial}
\end{align}
Moreover, if we assume that $\int_{\R}|U'(x)|^q e^{-U(x)}\d x < \infty$ for some $q > 1$, then there exists a SHM solution of the infinite SDE \eqref{DysonSDE} (with $\beta=2$), in which the joint density of two adjacent particles $(X^v_0,X^u_0)$ is given by
\begin{align}
\rho(x,y)&= \frac{1}{Z}(x-y)^2(x^2+r)^{m-1}(y^2+r)^{m-1} e^{-U(x)-U(y)}, \quad x,y \in \R, \ \ Z >0. \label{Dyson-density}
\end{align}
There exists also a stationary symmetric solution $(X,Y)$ of the local equation \eqref{Dyson-localeq}, such that $(X_0,Y_0)$ has the same joint density $\rho$.
\end{theorem}

\begin{remark}
In the case $m=2$, the polynomial \eqref{Dyson-polynomial} becomes the quadratic $s_0r^2 - s_4$, which has the unique positive root $r=\sqrt{s_4/s_0}$. For general $m \ge 2$, the coefficient $c$ depends on the moments of $e^{-U(y)}\d y$ up to order $2m$.
\end{remark}

\begin{proof}[Proof of Theorem \ref{th:Dyson}]
We begin by arguing uniqueness, showing that any solution of the fixed point problem must take the claimed form. Let $F$ be a solution of the fixed point problem with corresponding constant $C$, as in Definition \ref{def:fixedpoint}. Let $\lambda = e^C$, so that the fixed point equation \eqref{fixedpoint} becomes
\begin{align}
F(x) &= - \log \left(\frac{1}{\lambda} \int_\R (x-y)^2 e^{- U(y) - (m-1)F(y)} \d y\right). \label{pf:repulsion-F} 
\end{align}
That is, we have $F(x) = -\log(ax^2-2bx+c)$ for some $a,b,c \in \R$ for which $ax^2-2bx+c > 0$ for all $x \in \R$, or equivalently $a > 0$ and $ac > b^2$. Then, \eqref{pf:repulsion-F} is equivalent to
\begin{align*}
ax^2-2bx+c = e^{-F(x)} &= \frac{1}{\lambda}\int_\R (x-y)^2 e^{- U(y) - (m-1)F(y)} \d y \\
	&= \frac{1}{\lambda}\Big(x^2\int_\R e^{- U(y) - (m-1)F(y)} \d y - 2x \int_\R y \, e^{- U(y) - (m-1)F(y)} \d y \\
	&\qquad\quad + \int_\R y^2 e^{- U(y) - (m-1)F(y)} \d y\Big),
\end{align*}
for all $x \in \R$.
Match coefficients and plug in $e^{-F(y)} = ay^2-2by+c$ to find
\begin{align}
\lambda a &= \int_\R e^{- U(y) - (m-1)F(y)} \d y = \int_\R (ay^2-2by+c)^{m-1} e^{- U(y)} \d y \nonumber \\
\lambda b &= \int_\R y e^{- U(y) - (m-1)F(y)} \d y = \int_\R y (ay^2-2by+c)^{m-1} e^{- U(y)} \d y \label{pf:repulsion2} \\
\lambda c &= \int_\R y^2 e^{- U(y) - (m-1)F(y)} \d y = \int_\R y^2 (ay^2-2by+c)^{m-1} e^{- U(y)} \d y \nonumber .
\end{align}

We first argue that $b=0$ necessarily, which is natural; we expect $F$ to be even because $U$ is. Multiply the first equation in \eqref{pf:repulsion2} by $b$ and subtract it from $a$ times the second equation to get
\begin{align*}
0 &= \int_\R (ay-b) (ay^2-2by+c)^{m-1} e^{- U(y)} \d y. 
\end{align*}
Then, use the multinomial theorem and recall that $s_k := \int_\R x^k e^{-U(y)}\,\d y$ to get
\begin{align*}
0 &=  \int_\R (ay-b)\sum_{i+j+k=m-1} \frac{(m-1)!}{i!j!k!} a^i y^{2i}(-2by)^j c^k e^{- U(y)} \d y \\
	&= \sum_{i+j+k=m-1} \frac{(m-1)!}{i!j!k!} \int_\R \left(  y^{2i+j+1}a^{i+1}(-2b)^j c^k + \frac12  y^{2i+j}a^i(-2b)^{j+1} c^k \right) e^{- U(y)} \d y \\
	&= \sum_{i+j+k=m-1} \frac{(m-1)!}{i!j!k!} \left(s_{2i+j+1}a^{i+1}(-2b)^j c^k + \frac12  s_{2i+j}a^i(-2b)^{j+1} c^k \right),
\end{align*}
where the sum is over nonnegative integers $(i,j,k)$ which sum to $m-1$.
We may express this equation as $0=\sum_{j=0}^m r_j (-2b)^j$, for some coefficients $r_0,\ldots,r_m \in \R$ which depend on $a$, $c$, and $(s_0,\ldots,s_{2m-1})$.
Evenness of $U$ implies that $s_\ell=0$ for every odd number $\ell > 0$. It follows that $r_j=0$ for $j\ge 0$ even, and also $r_j > 0$ for $j$ odd, since $a,c > 0$. This proves that $b=0$ necessarily.

With $b=0$, we have $F(x)= -\log(ax^2+c)$, where $(\lambda,a,c)$ solve the first and third equations of \eqref{pf:repulsion2}, which simplify to
\begin{align}
\lambda a &= \int_\R (ay^2+c)^{m-1} e^{- U(y)} \d y \label{pf:repulsion-lambda1} \\
\lambda c &= \int_\R y^2 (ay^2+c)^{m-1} e^{- U(y)} \d y. \label{pf:repulsion-lambda2}
\end{align}
Multiply the first equation by $c$ and subtract it from $a$ times the second equation to get 
\begin{align*}
\int_\R (ay^2-c)(ay^2 + c)^{m-1} e^{-U(y)}\d y = 0,
\end{align*}
Letting $r=c/a$, dividing by $a^m$ yields
\begin{align}
\int_\R (y^2-r)(y^2 + r)^{m-1} e^{-U(y)}\d y = 0, \label{eq:momentpoly1}
\end{align}
The binomial theorem gives
\begin{align*}
(y^2-r)(y^2+r)^{m-1} &= (y^2-r)\sum_{j=0}^{m-1} \binom{m-1}{j}y^{2j} r^{m-j-1} \\
	&= \sum_{j=0}^{m-1} \binom{m-1}{j}y^{2(j+1)}r^{m-j-1} - \sum_{j=0}^{m-1} \binom{m-1}{j}y^{2j} r^{m-j} \\
	&= \sum_{j=1}^{m} \binom{m-1}{j-1}y^{2j}r^{m-j} - \sum_{j=0}^{m-1} \binom{m-1}{j}y^{2j}r^{m-j} \\
	&= y^{2m} - r^{m} - \sum_{j=1}^{m-1} \left(\binom{m-1}{j} - \binom{m-1}{j-1}\right)y^{2j}r^{m-j} \\
	&= y^{2m} - r^{m} - \sum_{j=1}^{m-1} \frac{m - 2j}{m}\binom{m}{j}y^{2j}r^{m-j} \\
	&= - \sum_{j=0}^{m} \frac{m - 2j}{m}\binom{m}{j}y^{2j}r^{m-j}.
\end{align*}
Plug this into \eqref{eq:momentpoly1}, recalling the definition of $s_j$, to find that $r$ must be a root of the polynomial \eqref{Dyson-polynomial}. The sequence of coefficients of $r^m,r^{m-1},\ldots,r^1,r^0$ in this polynomial switches signs exactly once, when $j$ crosses $m/2$ in the summation. Hence, by Descartes' rule of signs, there is a unique positive real root of \eqref{Dyson-polynomial}. Recalling that $r=c/a$, we get $F(x)=-\log(ax^2 +ar)$, which is an additive shift of \eqref{Dyson-Fsolution}.

This proves uniqueness of the fixed point problem, and reversing the argument shows that $F(x)$ given by \eqref{Dyson-Fsolution} is indeed a solution of the fixed point problem. To see this, note that the above calculations show that the unique positive real root $r$ of \eqref{Dyson-polynomial} yields a solution of the equation \eqref{eq:momentpoly1}. Defining $\lambda := \int_\R (y^2+r)^{m-1} e^{-U(y)}\d y$, it follows from \eqref{eq:momentpoly1} that \eqref{pf:repulsion-lambda2} holds as well, with $a=1$ and $c=r$. With $b=0$, we find that \eqref{pf:repulsion2} holds, from which we can easily deduce that $F(x) = -\log(x^2+r)$ satisfies \eqref{pf:repulsion-F}.

We finally prove the existence of a SHM solution by applying Theorem \ref{th:fixedpoint->infSDE}. With $F(x)=-\log(x^2+r)$, the integral in \eqref{asmp:tech1'} becomes
\begin{align*}
\int_{\R^2}\big(|U'(x)|^p + 2^p|x-y|^{-p}\big)(x-y)^2(x^2+r)^{m-1}(y^2+r)^{m-1} e^{-U(x)-U(y)}\d y \d x.
\end{align*}
This integral is finite for any $1 < p \le q \wedge 2$, by H\"older's inequality, thanks to the assumption that $\int_{\R}|U'(x)|^q e^{-U(x)}\d x < \infty$ for some $q > 1$ and that $e^{-U(y)}\d y$ has finite moments of every order. Hence, Theorem \ref{th:fixedpoint->infSDE} applies, yielding the desired SHM solution. Applying Theorem \ref{th:infSDE->local} yields the existence for the local equation.
\end{proof}

\bibliographystyle{amsalpha}
\bibliography{main}

\end{document}